\documentclass[12pt]{amsart}

\usepackage[pdfauthor   = {Mohamed\ Barakat\ and\ Markus\ Lange-Hegermann},
            pdftitle    = {Gabriel\ morphisms\ and\ the\ computability\ of\ Serre\ quotients},
            pdfsubject  = {},
            pdfkeywords = {Constructive\ homological\ algebra;\ Serre\ quotient\ category;\ generalized\ morphisms;\ Gabriel\ morphisms;\ commutative\ inverse\ monoids;\ zeroid;\ constructive\ coherent\ sheaves;\ Gröbner\ bases;\ homalg},
            bookmarks=true,
            bookmarksopen=true,
            pagebackref=true,
            hyperindex=true,
            colorlinks=true,
            linkcolor=blue,
            citecolor=blue,
            filecolor=blue,
            urlcolor=blue,
            ]{hyperref}

\usepackage[utf8,utf8x]{inputenc}
\usepackage[english]{babel}
\usepackage[T1]{fontenc}
\usepackage{geometry}                
\usepackage{times}
\usepackage[novbox]{pdfsync}
\usepackage{mathrsfs}
\usepackage{latexsym}
\usepackage{amssymb}
\usepackage{amsthm}
\usepackage{epsfig}
\usepackage{colortbl}
\usepackage[all]{xy}
\usepackage{fancyvrb}
\usepackage{graphicx}
\usepackage[dvipsnames]{xcolor}
\usepackage{accents} 
\usepackage{enumerate}

\renewcommand\theenumi{\alph{enumi}}
\renewcommand{\labelenumi}{(\theenumi)}

\usepackage{wrapfig}
\usepackage{tikz}
\usetikzlibrary{shapes,arrows,matrix,backgrounds,positioning,plotmarks,calc,patterns,matrix,
decorations.shapes,
decorations.fractals,
decorations.markings,
decorations.pathreplacing,
decorations.pathmorphing,
decorations.text
}
\usepackage[colorinlistoftodos,shadow]{todonotes}

\usepackage{multirow}
\usepackage{mdwlist}
\usepackage{stmaryrd}
\usepackage{mathdots} 

\usepackage[toc,page]{appendix}
\usepackage{float}

\newtheoremstyle{mytheoremstyle} 
    {5pt}                    
    {5pt}                    
    {\itshape}                   
    {\parindent}                           
    {\bf}                   
    {.}                          
    {.5em}                       
    {}  

\theoremstyle{mytheoremstyle}

\newtheorem{theorem}{Theorem}[section]

\newtheorem{lemm}[theorem]{Lemma}
\newtheorem{prop}[theorem]{Proposition}
\newtheorem{coro}[theorem]{Corollary}

\newtheoremstyle{mytdefintionstyle} 
    {5pt}                    
    {5pt}                    
    {\rm}                   
    {\parindent}                           
    {\bf}                   
    {.}                          
    {.5em}                       
    {}  

\theoremstyle{remark}
\newtheorem{rmrk}[theorem]{Remark}

\theoremstyle{mytdefintionstyle}
\newtheorem{defn}[theorem]{Definition}
\newtheorem{exmp}[theorem]{Example}

\newtheoremstyle{exmp_contd}
    {5pt}                    
    {5pt}                    
    {\rm}                   
    {\parindent}                           
    {\bf}                   
    {.}                          
    {.5em}                       
    {\thmname{#1}\ \thmnumber{ #2}\thmnote{#3}\ (continued)}  
\theoremstyle{exmp_contd}

\newcommand\nameft\textrm

\newcommand{\QQ}{{\mathscr{Q}}}
\renewcommand{\SS}{{\mathscr{S}}}

\newcommand{\WW}{{\mathscr{W}}}

\DeclareMathOperator{\Ann}{Ann}

\DeclareMathOperator{\Ext}{Ext}

\DeclareMathOperator{\coker}{coker}
\DeclareMathOperator{\Gen}{\mathbf{G}}
\DeclareMathOperator{\img}{im}

\DeclareMathOperator{\source}{source}

\DeclareMathOperator{\target}{target}
\DeclareMathOperator{\coimg}{coim}

\DeclareMathOperator{\Hom}{Hom}

\DeclareMathOperator{\Spec}{Spec}
\DeclareMathOperator{\Proj}{Proj}

\DeclareMathOperator{\Sat}{Sat}

\newcommand{\Sfpgrmod}{{\grS\mathrm{\textnormal{-}grmod}}}

\newcommand{\Coh}{\mathfrak{Coh}\,}

\newcommand\grS{S}

\newcommand\grM{M}
\newcommand\grN{N}

\newcommand\shF{\mathcal{F}}

\renewcommand\O{\mathcal{O}}
\newcommand\PP{\mathbb{P}}

\newcommand\A{\mathcal{A}}
\newcommand\C{\mathcal{C}}

\renewcommand{\O}{\mathcal{O}}
\newcommand{\Q}{\mathbb{Q}}
\newcommand{\CC}{\mathbb{C}}

\newcommand{\Z}{\mathbb{Z}}
\newcommand\N{\mathbb{N}}
\renewcommand\phi{\varphi}

\DeclareMathOperator\Id{Id}

\DeclareMathOperator\Cl{Cl}

\DeclareMathOperator\cores{co-res}

\definecolor{darkgray}{rgb}{0.3,0.3,0.3}

\newcommand{\liftofCalongB}[2]{{}^{#2} \negthinspace / \negthinspace {}_{#1}}
\newcommand{\coliftofCalongB}[2]{{}_{#1} \negthinspace \backslash \negthinspace {}^{#2}}

\newcommand{\closeoverline}[1]{
  \overline{{\raisebox{0pt}[\height-0.5pt\relax][0pt]{$#1$}}}
}

\newcommand{\closeunderline}[1]{
  \underline{{\raisebox{0pt}[0pt][\depth+1.5pt\relax]{$#1$}}}
}

\newcommand{\overunderline}[1]{
  \closeunderline{\closeoverline{#1}}
}

\pgfarrowsdeclarecombine[\pgflinewidth]
  {doublestealth}{doublestealth}{stealth'}{stealth'}{stealth'}{stealth'}

\newcounter{saveenum}

\definecolor{darkgreen}{rgb}{0.008,0.617,0.067}
\definecolor{brown}{rgb}{0.6,0.4,0.2}

\newif\ifjournalversion
\newcommand{\journal}[2]{\ifjournalversion #1\else #2\fi}

\author{Mohamed Barakat}
\address{Department of mathematics, University of Kaiserslautern, 67653 Kaiserslautern, Germany}
\email{\href{mailto:Mohamed Barakat <barakat@mathematik.uni-kl.de>}{barakat@mathematik.uni-kl.de}}

\author{Markus Lange-Hegermann}
\address{Lehrstuhl B f\"ur Mathematik, RWTH Aachen University, 52062 Aachen, Germany}
\email{\href{mailto:Markus Lange-Hegermann <markus.lange.hegermann@rwth-aachen.de>}{markus.lange.hegermann@rwth-aachen.de}}

\begin{document}

\title[Gabriel morphisms and the computability of Serre quotients]{Gabriel morphisms and the computability of Serre quotients with applications to coherent sheaves}

\begin{abstract}
  The purpose of this paper is to develop an efficient computational model for \nameft{Abel}\-ian categories of coherent sheaves over certain classes of varieties.
  These categories are naturally described as \nameft{Serre} quotient categories.
  Hence, our approach relies on describing general \nameft{Serre} quotient categories in a constructive way which leads to an efficient computer implementation.
\end{abstract}


\maketitle


\renewcommand\theenumi{\alph{enumi}}
\renewcommand{\labelenumi}{(\theenumi)}

\hyphenation{co-do-main}

\section{Introduction}

This paper develops an efficient computational model for \nameft{Abel}ian categories of coherent sheaves over certain classes of varieties.
The standard approach to constructive mathematics is to require that all disjunctions and all existential quantifiers appearing in defining axioms have to be realized by algorithms.
We call an \nameft{Abel}ian category for which this holds \textbf{constructively \nameft{Abel}ian} or \textbf{computable} (cf.~Appendix~\ref{sec:ComputableAbelian}, in particular Definition~\ref{defn:computable}).

Such categories of coherent sheaves are \nameft{Serre} quotient categories $\A/\C$ of some category $\A=\Sfpgrmod$ of finitely generated graded $S$-modules (over a graded \nameft{Noether}ian ring $S$) modulo a subcategory $\C$ of negligible modules, i.e., those with zero sheafification \cite{BL_SerreQuotients}.
Luckily, due to \nameft{Gröbner} bases methods such categories $\A$ of modules are constructively \nameft{Abel}ian \cite{BL}.
In the typical case of $\C$ not being a zero category one cannot naively model coherent sheaves by f.g.~graded $S$-modules.
Taking this discrepancy $\C$ into account is unavoidable for the correctness of many homological algorithms, e.g., those computing connecting homomorphisms or spectral sequences, as morphisms that should lift as morphisms of sheaves do not necessarily lift if they are naively modeled by $S$-module maps.
The obstructions to such lifts in $\A$ lie in the subcategory $\C$ and vanish in $\A/\C$ (cf.~Remark~\ref{rmrk:C_as_obstruction_to_lift}).

Our main result is to deduce the computability of $\A/\C$ from that of $\A$.

\begin{theorem} \label{thm:gabriel.computable}
  Let $\A$ be a constructively \nameft{Abel}ian category and $\C$ a thick subcategory.
  Assume that we can decide whether an object of $\A$ lies in $\C$.
  Then the category $\A/\C$ is constructively \nameft{Abel}ian.
\end{theorem}

This implies the following two corollaries which were the original motivation of this work.

\begin{theorem} \label{thm:proj.computable}
  Let $B$ be a computable commutative unitial ring with effective coset representatives (cf.~Definition~\ref{defn:coset.representatives}).
  Then the category $\Coh \PP^n_B$ is constructively \nameft{Abel}ian.
\end{theorem}

Examples for such a ring $B$ are fields, rings of integers, or polynomial rings.

\begin{theorem} \label{thm:toric.computable}
  Let $X_\Sigma$ be a smooth toric variety without torus factors.
  Then the category $\Coh X_\Sigma$ is constructively \nameft{Abel}ian.
\end{theorem}

In Section~\ref{sec:gabriel_morphism} we introduce the notion of generalized and \nameft{Gabriel} morphisms.
They are a $3$-arrow formalism for localization, already implicit in \cite{Tohoku,Gab_thesis}, and a computer friendly data structure.
In Section~\ref{sec:computability_quotient} we use \nameft{Gabriel} morphisms to provide a constructive description of morphisms in \nameft{Serre} quotient categories.
This leads to a proof of Theorem~\ref{thm:gabriel.computable}, which directly describes a computer implementation of \nameft{Serre} quotient categories.
In Section~\ref{sec:sheaves} we apply the abstract results for \nameft{Serre} quotient categories to categories of sheaves on projective and toric varieties and prove Theorems~\ref{thm:proj.computable} and \ref{thm:toric.computable}.
We demonstrate that our approach is practical by a spectral sequence computation in Appendix~\ref{sec:ex} using our computer implementation in the \texttt{homalg} project \cite{homalg-project}.

Recently, a similar $3$-arrow formalism was introduced for \nameft{Quillen} model categories admitting a functorial factorization \cite{dhks}, and a more general $3$-arrow formalism was developed for uni-fractional categories by \nameft{S.~Thomas}.
The novelty of our formalism is that it constructs a new category $\Gen_\C(\A)$ and recovers $\A/\C$ (also treated in \cite[Example~7.7.b]{ThomasArrow}) by factoring out a two-sided ideal which we call the zeroid, following \nameft{Bourne} and \nameft{Zassenhaus}.
When considered in a constructive setup \nameft{Thomas} verified that his formalism defines a category (with the exception of deciding equality of morphisms), whereas in our proof of Theorem~\ref{thm:gabriel.computable} we constructively verify all axioms of an \nameft{Abel}ian category.

In the above discussion we didn't assume the existence of a section functor $\SS: \A/\C \to \A$, right adjoint to the exact localization functor $\QQ:\A \to \A/\C$.
In Appendix~\ref{sec:reflective} we prove alternative versions of Theorem~\ref{thm:gabriel.computable} once $\SS$ is computable, for example using computations in the essential image of $\SS$ or using the \nameft{Gabriel}-\nameft{Zisman}-localization \cite{GabZis,SimpsonGZComputer}.
However, computer experiments suggest that these alternative approaches are less efficient than the approach using \nameft{Gabriel} morphisms.
Still, the section functor plays a major role in our treatment of the computability of $\Hom$ and $\Ext$ in \nameft{Serre} quotients categories \cite{BL_Sheaves,BL_ExtComputability}.

\medskip
\textbf{Convention:} In this paper, we use the postfix notation for the composition, i.e., $fg:=g\circ f$.

\section{Generalized morphisms and \nameft{Gabriel} morphisms}\label{sec:gabriel_morphism}

In this section we introduce the $3$-arrow calculus of generalized morphisms.
This paper uses a subclass of generalized morphisms, called \nameft{Gabriel} morphisms, as the ``data structure'' for morphisms in a \nameft{Serre} quotient category $\A/\C$.
These morphisms are the key ingredient to prove the computability of $\A/\C$ as an \nameft{Abel}ian category in Section~\ref{sec:computability_quotient}.

There is a second application of generalized morphisms in constructive homological algebra.
They provide a simple constructive description of spectral sequences of filtered complexes where the differentials of the successive pages were defined by closed formulas \cite{BaSF}.

Arrow diagrams of morphisms in \nameft{Abel}ian categories induce isomorphisms between various subfactors of the involved objects.
To visualize these isomorphisms we occasionally display the corresponding \nameft{Hasse} diagrams of modular subobject lattices together with the order-preserving maps induced by the morphisms.
The zero subobject is located at the top of each \nameft{Hasse} diagram.
Subfactor objects correspond to intervals, subobjects to intervals starting at the top, and factor objects to intervals ending at the bottom.
For more details see \cite{BaSF}.

\subsection{Generalized morphisms} 

From now on $\A$ denotes an \nameft{Abel}ian category.

\begin{defn}[\cite{BaSF}]
  Let $M,N$ be two objects in $\A$.
  A \textbf{generalized morphism} of $\A$ with \textbf{source} $M$ and \textbf{target} $N$ is an equivalence class of triples of $\A$-morphisms\footnote{Necessarily $N_\psi \cong N/\ker \jmath_\psi$.}
  \[
    \psi:=[M \stackrel{\imath_\psi}{\hookleftarrow} M^\psi \stackrel{\overunderline{\mbox{\tiny $\psi$}}}{\longrightarrow} N_\psi \stackrel{\jmath_\psi}{\twoheadleftarrow} N]
  \]
  such that the \textbf{domain} $\imath_\psi:M^\psi \hookrightarrow M$ \textbf{of $\psi$} is a mono and the \textbf{codomain} $\jmath_\psi:N \twoheadleftarrow N_\psi$ \textbf{of $\psi$} is an epi.
  We call $\overunderline{\psi}$ the $\A$-morphism \textbf{associated} to $\psi$.
  Two such triples are equivalent if
  \begin{itemize}
    \item their sources and targets \emph{coincide} and
    \item $(M \stackrel{\imath_\psi}{\hookleftarrow} M^\psi \stackrel{\overunderline{\mbox{\tiny $\psi$}}}{\longrightarrow} N_\psi \stackrel{\jmath_\psi}{\twoheadleftarrow} N) \sim (M \stackrel{\imath_\phi}{\hookleftarrow} M^\phi \stackrel{\overunderline{\mbox{\tiny $\phi$}}}{\longrightarrow} N_\phi \stackrel{\jmath_\phi}{\twoheadleftarrow} N)$ if there exists isomorphisms $\mu: M^\phi \to M^\psi$ and $\nu:N_\phi \to N_\psi$ (then they are necessarily the unique lift $\mu = \liftofCalongB{\imath_\psi}{\imath_\phi}$ and the unique colift $\nu = \coliftofCalongB{\jmath_\phi}{\jmath_\psi}$, cf.~\ref{sec:ComputableAbelian}.\eqref{Ab_DecideZeroEffectively},\eqref{Ab_colift}) such that the following diagram commutes
  \begin{center}
  \begin{tikzpicture}[label/.style={postaction={
      decorate,
      decoration={markings, mark=at position .5 with \node #1;}}},scale=0.8]
    
    \coordinate (r) at (4.5,0);
    \coordinate (u) at (0,1.6);
    
    \node (M) {$M$};
    \node (N) at (r) {$N$};
    \node (Mpsi) at ($-1*(u)$) {$M^\psi$};
    \node (Npsi) at ($(N)-(u)$) {$N_\psi$};
    \node (Mphi) at (u) {$M^\phi$};
    \node (Nphi) at ($(N)+(u)$) {$N_\phi$};
    
    \draw[-stealth'] (Mpsi) -- node[above]{$\overunderline{\psi}$} (Npsi);
    \draw[left hook-stealth'] (Mpsi) -- node[left] {$\imath_\psi$} (M);
    \draw[-doublestealth] (N) -- node[right] {$\jmath_\psi$} (Npsi);
    \draw[-stealth'] (Mphi) -- node[above]{$\overunderline{\phi}$} (Nphi);
    \draw[right hook-stealth'] (Mphi) -- node[left] {$\imath_\phi$} (M);
    \draw[-doublestealth] (N) -- node[right] {$\jmath_\phi$} (Nphi);
    
    \draw[bend left=60,-stealth',label={[right]{$\mu$}}] (Mphi) to (Mpsi);
    \draw[bend right=60,-stealth',label={[left]{$\nu$}}] (Nphi) to (Npsi);
    
  \end{tikzpicture}
  \end{center}
  \end{itemize}

For such an equivalence class we write more elaborately
\begin{center}
\begin{tikzpicture}[C/.style={color=red,line width=2pt,dotted},align=left,
  map/.style={color=gray,dashed},behind/.style={opacity=0.5},image/.style={color=blue,line width=3},
  image/.style={color=blue,line width=3},
  scale=0.8]
  
  \node (M) {$M$};
  \node (N) at (2,0) {$N$};
  \node (Mpsi) at (0,-2) {$M^\psi$};
  \node (Npsi) at (2,-2) {$N_\psi$};
  
  \draw[-stealth'] (Mpsi) -- node[above]{$\overunderline{\psi}$} (Npsi);
  \draw[left hook-stealth'] (Mpsi) -- node[left] {$\imath_\psi$} (M);
  \draw[-doublestealth] (N) -- node[right] {$\jmath_\psi$} (Npsi);
  \draw[dotted,-stealth'] (M) -- node[above]{$\psi$} (N);

\begin{scope}[xshift=13em,yshift=-2em]
  \node {with corresponding \\ \nameft{Hasse} diagram of $\overunderline{\psi}$};
\end{scope}

\begin{scope}[xshift=25em,yshift=2em]
  \coordinate (a) at (0,-0.9);
  \coordinate (b) at ($2*(a)$);
  \coordinate (c) at ($0.75*(a)$);
  \coordinate (d) at (2,0);
  \coordinate (e) at ($1.2*(a)$);
  \coordinate (f) at ($0.9*(a)$);
  \coordinate (right) at (0.2,0);
  \coordinate (left) at ($-1*(right)$);
   
  \coordinate (0_M);
  \coordinate (ker_phi) at ($(0_M)+(a)$);
  \coordinate (domain) at ($(ker_phi)+(b)$);
  \coordinate (M) at ($(domain)+(c)$);
  \coordinate (codomain) at ($(ker_phi)+(d)$);
  \coordinate (0_N) at ($(codomain)-(f)$);
  \coordinate (im_phi) at ($(codomain)+(b)$);
  \coordinate (N) at ($(im_phi)+(e)$);
  
  
  \draw (0_M) -- (ker_phi);
  \draw[image] (ker_phi) -- (domain);
  \draw[image] (codomain) -- (im_phi);
  \draw (im_phi) -- (N);
  
  \draw[C] (domain) -- (M);
  \draw[C] (0_N) -- (codomain);
  
  \draw[map,-stealth'] ($(domain)+(right)$) -- ($(im_phi)+(left)$);
  \draw[map,-stealth'] ($(ker_phi)+(right)$) -- ($(codomain)+(left)$);
  
  \fill (0_M) circle (2pt) node[left] {$0_M$};
  \fill (ker_phi) circle (2pt) node[left] {$\ker\psi$};
  \fill (domain) circle (2pt) node[left] {$M^\psi\cong\img\imath_\psi$};
  \fill (M) circle (2pt) node[left] {$M$};
  \fill (codomain) circle (2pt) node[right] {$\ker \jmath_\psi$};
  \fill (0_N) circle (2pt) node[right] {$0_N$};
  \fill (im_phi) circle (2pt) node[right] {$\img\psi$};
  \fill (N) circle (2pt) node[right] {$N$};
  
  \node at ($0.2*(ker_phi)+0.3*(im_phi)+0.2*(codomain)+0.3*(domain)$) {\color{gray} $\overunderline{\psi}$};
  
  \draw[decorate,decoration=brace] ($(codomain)+(1.5,0)$) -- node[right] {$\cong N_\psi$} ($(N)+(1.5,0)$);
  
\end{scope}
\end{tikzpicture}
\end{center}
The subfactors isomorphic to the image ($\cong$ coimage) of $\overunderline{\psi}$ are marked by thick (blue) lines.
The cokernel of the domain $\imath_\psi$ and the kernel of the codomain $\imath_\psi$ are marked by dotted (red) lines.

We call a domain or codomain \textbf{full} if it is an isomorphism\footnote{This does not depend on the representing triple.}, and non-full otherwise.
A generalized morphism with full domain and codomain is called \textbf{honest}\index{morphism!honest}.
Every honest morphism $\psi:=[M \xleftarrow[{}^\sim]{\imath_\psi} M^\psi \stackrel{\overunderline{\mbox{\tiny $\psi$}}}{\longrightarrow} N_\psi \xleftarrow[{}^\sim]{\jmath_\psi} N]$ has a unique representative $M \stackrel{1_M}{\longleftarrow} M \xrightarrow{\imath_\psi^{-1}\, \overunderline{\mbox{\tiny $\psi$}}\, \jmath_\psi^{-1}} N \stackrel{1_N}{\longleftarrow} N$.
Conversely, for any $\A$-morphism $\phi:M \to N$ we call $[M \stackrel{1_M}{\longleftarrow} M \stackrel{\phi}{\to} N \stackrel{1_N}{\longleftarrow} N]$ the honest morphism \textbf{induced by $\phi$}.

We write $(\imath_\psi,\overunderline{\psi},\jmath_\psi):M \to N$ as a shorthand for the triple $M \stackrel{\imath_\psi}{\hookleftarrow} M^\psi \stackrel{\overunderline{\mbox{\tiny $\psi$}}}{\longrightarrow} N_\psi \stackrel{\jmath_\psi}{\twoheadleftarrow} N$ and $[\imath_\psi,\overunderline{\psi},\jmath_\psi]$ for the generalized morphism $[(\imath_\psi,\overunderline{\psi},\jmath_\psi)]$.
\end{defn}

\subsection{The category \texorpdfstring{$\Gen(\A)$}{G(A)} of generalized morphisms}

We start with defining the composition of generalized morphisms, following \cite[III.1]{Gab_thesis}.

\begin{defn} \label{defn:composition_generalized}
Let $\phi:=[\imath_\phi,\overunderline{\phi},\jmath_\phi]:L\to M$ and $\psi:=[\imath_\psi,\overunderline{\psi},\jmath_\psi]:M\to N$ be two generalized morphisms of $\A$.
The morphism $\alpha:=\imath_\psi\jmath_\phi:M^\psi \to M_\phi$ naturally gives rise to the epi $\jmath:M^{\psi} \twoheadrightarrow \coimg \alpha$, the mono $\imath:\img \alpha \hookrightarrow M_\phi$, and the isomorphism $\widetilde{\alpha}:\coimg \alpha \to \img \alpha$.
Denote by $\widetilde{\imath}$ and $\widetilde{\phi}$ the pullback morphisms of the cospan $(\imath,\overunderline{\phi})$, as indicated in the commutative diagram below.
Analogously, denote the pushout morphisms of the span $(\jmath,\overunderline{\psi})$ by $\widetilde{\jmath}$ and $\widetilde{\psi}$.
\begin{center}
\begin{tikzpicture}[C/.style={color=red,line width=1.5pt,dotted}]
  \coordinate (right) at (3.9,0);
  \coordinate (above) at (0,1.4);
    \coordinate (_Ma) at (0,0);
    \node (Ma) at (_Ma) {$M$};
    \coordinate (_Mb) at ($(_Ma)+(right)$);
    \node (Mb) at (_Mb) {$M$};
    \coordinate (_L) at ($(_Ma)-(right)$);
    \node (L) at (_L) {$L$};
    \coordinate (_N) at ($(_Mb)+(right)$);
    \node (N) at (_N) {$N$};
    \coordinate (_L1) at ($(_L)-(above)$);
    \node (L1) at (_L1) {$L^{\phi}$};
    \coordinate (_pullback) at ($(_L1)-(above)$);
    \node (pullback) at (_pullback) {$L^{\phi \psi} := \operatorname{pullback}(\imath,\overunderline{\phi})$};
    \coordinate (_M1) at ($(_Ma)-(above)$);
    \node (M1) at (_M1) {$M_{\phi}$};
    \coordinate (_M2) at ($(_Mb)-(above)$);
    \node (M2) at (_M2) {$M^{\psi}$};
    \coordinate (_imbeta) at ($(_M1)-(above)$);
    \node (imbeta) at (_imbeta) {$\img\alpha$};
    \coordinate (_coimbeta) at ($(_M2)-(above)$);
    \node (coimbeta) at (_coimbeta) {$\coimg\alpha$};
    \coordinate (_N2) at ($(_N)-(above)$);
    \node (N2) at (_N2) {$N_{\psi}$};
    \coordinate (_pushout) at ($(_N2)-(above)$);
    \node (pushout) at (_pushout) {$\operatorname{pushout}(\jmath,\overunderline{\psi}) =: N_{\phi \psi}$};
    \draw[double equal sign distance,gray] (Ma)     -- node[above] {$1_M$}                                     (Mb);
    \draw[right hook-stealth',thick] (L1)                    -- node[left] {$\imath_\phi$}                     (L);
    \draw[right hook-stealth',thick] (pullback)              -- node[left] {$\widetilde{\imath}$}              (L1);
    \draw[-stealth',thick] (pullback)                       -- node[below] {$\widetilde{\phi}$}               (imbeta);
    \draw[-doublestealth,gray] (Ma)                            -- node[right] {$\jmath_\phi$}                    (M1);
    \draw[right hook-stealth',gray] (M2)                    -- node[left] {$\imath_\psi$}                     (Mb);
    \draw[right hook-stealth',gray] (imbeta)                -- node[right] {$\imath$}                         (M1);
    \draw[-doublestealth,gray] (M2)                            -- node[left] {$\jmath$}                          (coimbeta);
    \draw[-doublestealth,thick] (N)                             -- node[right] {$\jmath_\psi$}                    (N2);
    \draw[-doublestealth,thick] (N2)                            -- node[right] {$\widetilde{\jmath}$}             (pushout);
    \draw[-stealth',gray]  (L1)                            -- node[above] {$\overunderline{\phi}$}    (M1);
    \draw[-stealth',gray]  (M2)                            -- node[above] {$\overunderline{\psi}$}    (N2);
    \draw[-stealth',gray]  (M2)                            -- node[above] {$\alpha$}                         (M1);
    \draw[-stealth',thick]  (coimbeta)                      -- node[below] {$\widetilde{\psi}$}               (pushout);
    \draw[-stealth',thick] (imbeta)                         -- node[above] {$\sim$} node[below] {$\left(\widetilde{\alpha}\right)^{-1}$} (coimbeta);
    \draw[-stealth',dotted,gray] (L)                       -- node[above] {$\phi$}                           (Ma);
    \draw[-stealth',dotted,gray] (Mb)                      -- node[above] {$\psi$}                           (N);
\end{tikzpicture}
\end{center}
Finally define the \textbf{composition of the generalized morphisms} $\phi,\psi$ as
\[
  \phi\psi = \left[ \imath_{\phi \psi}, \overunderline{\phi \psi}, \jmath_{\phi \psi} \right]
  :=
  \left[\widetilde{\imath}\,\imath_\phi, \big(\widetilde{\phi}\left(\widetilde{\alpha}\right)^{-1}\widetilde{\psi}\big),\jmath_\psi\widetilde{\jmath}\right]:L \to N \mbox{,}
\]
i.e., the compositions of the outer left, bottom, and right morphisms in the above diagram.
\end{defn}

\journal{}{
The four morphisms $L^\phi \stackrel{\overline{\underline{\mbox{\tiny $\phi$}}}}{\longrightarrow} M_\phi \twoheadleftarrow M \hookleftarrow M^\psi \stackrel{\overline{\underline{\mbox{\tiny $\psi$}}}}{\longrightarrow} N_\psi$ are displayed in the diagram below, where $L^\phi$ is identified with a subobject of $L$ and $N_\psi$ with a factor object of $N$.
The subfactors isomorphic to the image ($\cong$ coimage) of the $\A$-morphism $\overunderline{\phi \psi}$ are marked by thick (blue) lines.
Subfactors of cokernels of domains and subfactors of kernels of codomains are marked by dotted lines.
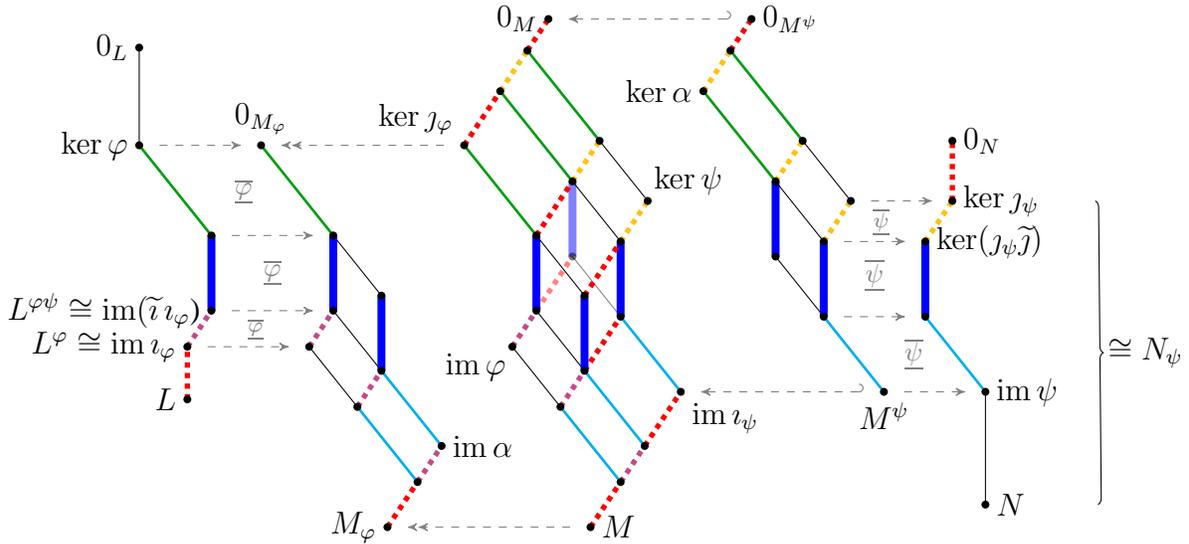
\begin{center}
\vskip -0.2em
\begin{tikzpicture}[
  C/.style={color=red,line width=2pt,dotted},
  C1/.style={color=magenta!80!black,line width=2pt,dotted},
  C2/.style={color=yellow!50!orange,line width=2pt,dotted},
  map/.style={color=gray,dashed},behind/.style={opacity=0.5},
  image1/.style={color=darkgreen,line width=1},
  image2/.style={color=cyan,line width=1},
  image/.style={color=blue,line width=3}]
  
  \newcommand{\circlewidth}{1.5pt}
  
  \coordinate (a) at (0.4,0.6);
  \coordinate (b) at ($0.8*(a)$);
  \coordinate (c) at ($1.2*(a)$);
  \coordinate (d) at (-0.8,1);
  \coordinate (e) at ($0.8*(d)$);
  \coordinate (f) at (0,1.0);
  \coordinate (g) at ($1.2*(d)$);
  \coordinate (h) at ($0.9*(a)$);
  \coordinate (i) at ($0.7*(a)$);
  \coordinate (j) at (2.7,0);
  \coordinate (k) at (0,-1);
  
  \coordinate (o);
  
  \coordinate (M) at (o);
  \coordinate (Z1+Z2) at ($(M)+(a)$);
  \coordinate (Z1+B2) at ($(Z1+Z2)+(d)$);
  \coordinate (Z1) at ($(Z1+B2)+(e)$);
  \coordinate (B1+Z2) at ($(Z1+Z2)+(b)$);
  \coordinate (bbb) at ($(B1+Z2)+(d)$);
  \coordinate (B1+B2) at ($(bbb)+(f)$);
  \coordinate (int_Z2) at ($(bbb)+(e)$);
  \coordinate (B1+int_B2) at ($(int_Z2)+(f)$);
  \coordinate (B1) at ($(B1+int_B2)+(g)$);
  \coordinate (Z2) at ($(B1+Z2)+(c)$);
  \coordinate (Z1_Z2+B2) at ($(Z2)+(d)$);
  \coordinate (Z1_Z2) at ($(Z1_Z2+B2)+(e)$);
  \coordinate (B2+int_B2) at ($(Z1_Z2+B2)+(f)$);
  \coordinate (int_B2) at ($(Z1_Z2)+(f)$);
  \coordinate (B2) at ($(B2+int_B2)+(h)$);
  \coordinate (B2_int_B2) at ($(int_B2)+(h)$);
  \coordinate (B1_Z2) at ($(int_B2)+(g)$);
  \coordinate (B1_B2) at ($(B1_Z2)+(h)$);
  \coordinate (0) at ($(B1_B2)+(i)$);
  
  
  \draw[C,behind] (int_Z2) -- (Z1_Z2);
  \draw[behind] (Z1_Z2+B2) -- (Z1_Z2);
  \draw[image,behind] (Z1_Z2) -- (int_B2);
  \draw[image] (int_Z2) -- (B1+int_B2);
  \draw[image] (Z1_Z2+B2) -- (B2+int_B2);
  
  \draw[C] (M) -- (Z1+Z2);
  \draw[C1] (Z1+B2) -- (bbb);
  \draw[C1] (Z1+Z2) -- (B1+Z2);
  \draw[C] (B1+Z2) -- (Z2);
  \draw[C] (bbb) -- (Z1_Z2+B2);
  \draw[C] (B1+B2) -- (B2+int_B2);
  \draw[C] (B1+int_B2) -- (int_B2);
  \draw[C2] (B1_Z2) -- (B1_B2);
  \draw[C] (B1) -- (B1_Z2);
  \draw[C2] (B2+int_B2) -- (B2);
  \draw[C2] (int_B2) -- (B2_int_B2);
  \draw[C1] (Z1) -- (int_Z2);
  \draw[C] (B1_B2) -- (0);
  
  \draw[image2] (Z1+Z2) -- (Z1+B2);
  \draw (Z1+B2) -- (Z1);
  \draw[image2] (B1+Z2) -- (bbb);
  \draw (bbb) -- (int_Z2);
  \draw (B1+B2) -- (B1+int_B2);
  \draw[image1] (B1+int_B2) -- (B1);
  \draw[image2] (Z2) -- (Z1_Z2+B2);
  \draw[image1] (int_B2) -- (B1_Z2);
  \draw[image1] (B2_int_B2) -- (B1_B2);
  \draw (B2+int_B2) -- (int_B2);
  \draw (B2) -- (B2_int_B2);
  \draw[image] (bbb) -- (B1+B2);
  
  \fill (M) circle (\circlewidth);
  \fill (Z1+Z2) circle (\circlewidth);
  \fill (Z1+B2) circle (\circlewidth);
  \fill (Z1) circle (\circlewidth);
  \fill (B1+Z2) circle (\circlewidth);
  \fill (bbb) circle (\circlewidth);
  \fill (B1+B2) circle (\circlewidth);
  \fill (int_Z2) circle (\circlewidth);
  \fill (B1+int_B2) circle (\circlewidth);
  \fill (B1) circle (\circlewidth);
  \fill (Z2) circle (\circlewidth);
  \fill (Z1_Z2+B2) circle (\circlewidth);
  \fill (B2+int_B2) circle (\circlewidth);
  \fill (int_B2) circle (\circlewidth);
  \fill (B2) circle (\circlewidth);
  \fill (B2_int_B2) circle (\circlewidth);
  \fill (B1_Z2) circle (\circlewidth);
  \fill (B1_B2) circle (\circlewidth);
  \fill (0) circle (\circlewidth);
  \fill[behind] (Z1_Z2) circle (\circlewidth);

  \node at (M) [right] {$M$};
  \node at (Z1) [below left=-0.2em] {$\img \phi$};
  \node at (B1) [above left] {$\ker\jmath_\phi$};
  \node at (Z2)[below right] {$\img \imath_\psi$};
  \node at (B2) [above right=-0.2em] {$\ker\psi$};
  \node at (0) [left] {$0_M$};
  
  
  \coordinate (M_psi) at ($(o)+(j)$);
  \coordinate (Z1+Z2_psi) at ($(M_psi)+(a)$);
  \coordinate (Z1+B2_psi) at ($(Z1+Z2_psi)+(d)$);
  \coordinate (Z1_psi) at ($(Z1+B2_psi)+(e)$);
  \coordinate (B1+Z2_psi) at ($(Z1+Z2_psi)+(b)$);
  \coordinate (bbb_psi) at ($(B1+Z2_psi)+(d)$);
  \coordinate (B1+B2_psi) at ($(bbb_psi)+(f)$);
  \coordinate (int_Z2_psi) at ($(bbb_psi)+(e)$);
  \coordinate (B1+int_B2_psi) at ($(int_Z2_psi)+(f)$);
  \coordinate (B1_psi) at ($(B1+int_B2_psi)+(g)$);
  \coordinate (Z2_psi) at ($(B1+Z2_psi)+(c)$);
  \coordinate (Z1_Z2+B2_psi) at ($(Z2_psi)+(d)$);
  \coordinate (Z1_Z2_psi) at ($(Z1_Z2+B2_psi)+(e)$);
  \coordinate (B2+int_B2_psi) at ($(Z1_Z2+B2_psi)+(f)$);
  \coordinate (int_B2_psi) at ($(Z1_Z2_psi)+(f)$);
  \coordinate (B2_psi) at ($(B2+int_B2_psi)+(h)$);
  \coordinate (B2_int_B2_psi) at ($(int_B2_psi)+(h)$);
  \coordinate (B1_Z2_psi) at ($(int_B2_psi)+(g)$);
  \coordinate (B1_B2_psi) at ($(B1_Z2_psi)+(h)$);
  \coordinate (0_psi) at ($(B1_B2_psi)+(i)$);
  
  
  \draw[C2] (B1_Z2_psi) -- (B1_B2_psi);
  \draw[C2] (B2+int_B2_psi) -- (B2_psi);
  \draw[C2] (int_B2_psi) -- (B2_int_B2_psi);
  \draw[C] (B1_B2_psi) -- (0_psi);
  
  \draw[image2] (Z2_psi) -- (Z1_Z2+B2_psi);
  \draw (Z1_Z2+B2_psi) -- (Z1_Z2_psi);
  \draw[image] (Z1_Z2_psi) -- (int_B2_psi);
  \draw[image1] (int_B2_psi) -- (B1_Z2_psi);
  \draw[image1] (B2_int_B2_psi) -- (B1_B2_psi);
  \draw[image] (Z1_Z2+B2_psi) -- (B2+int_B2_psi);
  \draw (B2+int_B2_psi) -- (int_B2_psi);
  \draw (B2_psi) -- (B2_int_B2_psi);
  
  \fill (Z2_psi) circle (\circlewidth);
  \fill (Z1_Z2+B2_psi) circle (\circlewidth);
  \fill (B2+int_B2_psi) circle (\circlewidth);
  \fill (Z1_Z2_psi) circle (\circlewidth);
  \fill (int_B2_psi) circle (\circlewidth);
  \fill (B2_psi) circle (\circlewidth);
  \fill (B2_int_B2_psi) circle (\circlewidth);
  \fill (B1_Z2_psi) circle (\circlewidth);
  \fill (B1_B2_psi) circle (\circlewidth);
  \fill (0_psi) circle (\circlewidth);
  
  \node at (Z2_psi) [below] {$M^\psi$};
  \node at (B1_Z2_psi) [left] {$\ker \alpha$};
  \node at (0_psi) [right] {$0_{M^\psi}$};
  
  \draw[map,left hook-stealth'] ($(0_psi)-0.1*(j)$) -- ($(0)+0.1*(j)$);
  \draw[map,left hook-stealth'] ($(Z2_psi)-0.1*(j)$) -- ($(Z2)+0.1*(j)$);
  
  
  \coordinate (im2) at ($(Z2_psi)+0.5*(j)$);
  \coordinate (U2) at ($(Z1_Z2+B2_psi)+0.5*(j)$);
  \coordinate (D2) at ($(B2+int_B2_psi)+0.5*(j)$);
  \coordinate (codomain2) at ($(B2_psi)+0.5*(j)$);
  
  \coordinate (N) at ($(im2)+1.5*(k)$);
  \coordinate (N0) at ($(codomain2)-0.8*(k)$);

  \draw (N) -- (im2);
  \draw[image2] (im2) -- (U2);
  \draw[image] (U2) -- (D2); 
  \draw[C2] (D2) -- (codomain2);
  \draw[C] (codomain2) -- (N0);
  
  \fill (N) circle (\circlewidth);
  \fill (im2) circle (\circlewidth);
  \fill (U2) circle (\circlewidth);
  \fill (D2) circle (\circlewidth);
  \fill (codomain2) circle (\circlewidth);
  \fill (N0) circle (\circlewidth);
  
  \node at (N) [right] {$N$};
  \node at (im2) [right] {$\img \psi$};
  \node at (D2) [right] {$\ker (\jmath_\psi \widetilde{\jmath})$};
  \node at (codomain2) [right] {$\ker \jmath_\psi$};
  \node at (N0) [right] {$0_N$};
  
  \draw[map,-stealth'] ($(Z2_psi)+0.1*(j)$) -- ($(im2)-0.1*(j)$);
  \draw[map,-stealth'] ($(Z1_Z2+B2_psi)+0.1*(j)$) -- ($(U2)-0.1*(j)$);
  \draw[map,-stealth'] ($(B2+int_B2_psi)+0.1*(j)$) -- ($(D2)-0.1*(j)$);
  \draw[map,-stealth'] ($(B2_psi)+0.1*(j)$) -- ($(codomain2)-0.1*(j)$);
  
  \node at ($0.5*(B2_psi)+0.5*(D2)-0.25*(h)$) {\color{gray} \tiny $\overunderline{\psi}$};
  \node at ($0.5*(B2+int_B2_psi)+0.5*(U2)-0.1*(f)$) {\color{gray} \footnotesize $\overunderline{\psi}$};
  \node at ($0.5*(Z1_Z2+B2_psi)+0.5*(im2)-0.15*(d)$) {\color{gray} \footnotesize $\overunderline{\psi}$};
  
  \draw[decorate,decoration=brace] ($(codomain2)+(1.9,0)$) -- node[right] {$\cong N_\psi$} ($(N)+(1.5,0)$);
  
  
  
  \coordinate (M_phi) at ($(o)-(j)$);
  \coordinate (Z1+Z2_phi) at ($(M_phi)+(a)$);
  \coordinate (Z1+B2_phi) at ($(Z1+Z2_phi)+(d)$);
  \coordinate (Z1_phi) at ($(Z1+B2_phi)+(e)$);
  \coordinate (B1+Z2_phi) at ($(Z1+Z2_phi)+(b)$);
  \coordinate (bbb_phi) at ($(B1+Z2_phi)+(d)$);
  \coordinate (B1+B2_phi) at ($(bbb_phi)+(f)$);
  \coordinate (int_Z2_phi) at ($(bbb_phi)+(e)$);
  \coordinate (B1+int_B2_phi) at ($(int_Z2_phi)+(f)$);
  \coordinate (B1_phi) at ($(B1+int_B2_phi)+(g)$);
  \coordinate (Z2_phi) at ($(B1+Z2_phi)+(c)$);
  \coordinate (Z1_Z2+B2_phi) at ($(Z2_phi)+(d)$);
  \coordinate (Z1_Z2_phi) at ($(Z1_Z2+B2_phi)+(e)$);
  \coordinate (B2+int_B2_phi) at ($(Z1_Z2+B2_phi)+(f)$);
  \coordinate (int_B2_phi) at ($(Z1_Z2_phi)+(f)$);
  \coordinate (B2_phi) at ($(B2+int_B2_phi)+(h)$);
  \coordinate (B2_int_B2_phi) at ($(int_B2_phi)+(h)$);
  \coordinate (B1_Z2_phi) at ($(int_B2_phi)+(g)$);
  \coordinate (B1_B2_phi) at ($(B1_Z2_phi)+(h)$);
  \coordinate (0_phi) at ($(B1_B2_phi)+(i)$);
  
  \draw[C] (M_phi) -- (Z1+Z2_phi);
  \draw[C1] (Z1+B2_phi) -- (bbb_phi);
  \draw[C1] (Z1+Z2_phi) -- (B1+Z2_phi);
  \draw[C1] (Z1_phi) -- (int_Z2_phi);
  
  \draw[image2] (Z1+Z2_phi) -- (Z1+B2_phi);
  \draw (Z1+B2_phi) -- (Z1_phi);
  \draw[image2] (B1+Z2_phi) -- (bbb_phi);
  \draw (bbb_phi) -- (int_Z2_phi);
  \draw[image] (bbb_phi) -- (B1+B2_phi);
  \draw (B1+B2_phi) -- (B1+int_B2_phi);
  \draw[image] (int_Z2_phi) -- (B1+int_B2_phi);
  \draw[image1] (B1+int_B2_phi) -- (B1_phi);
  
  \fill (M_phi) circle (\circlewidth);
  \fill (Z1+Z2_phi) circle (\circlewidth);
  \fill (Z1+B2_phi) circle (\circlewidth);
  \fill (Z1_phi) circle (\circlewidth);
  \fill (B1+Z2_phi) circle (\circlewidth);
  \fill (bbb_phi) circle (\circlewidth);
  \fill (B1+B2_phi) circle (\circlewidth);
  \fill (int_Z2_phi) circle (\circlewidth);
  \fill (B1+int_B2_phi) circle (\circlewidth);
  \fill (B1_phi) circle (\circlewidth);

  \node at (M_phi) [left] {$M_\phi$};
  \node at (B1+Z2_phi) [right] {$\img \alpha$};
  \node at (B1_phi) [above] {$0_{M_\phi}$};
  
  
  \coordinate (domain1) at ($(Z1_phi)-0.6*(j)$);
  \coordinate (U1) at ($(int_Z2_phi)-0.6*(j)$);
  \coordinate (D1) at ($(B1+int_B2_phi)-0.6*(j)$);
  \coordinate (ker1) at ($(B1_phi)-0.6*(j)$);
  
  \coordinate (L) at ($(domain1)+0.7*(k)$);
  \coordinate (L0) at ($(ker1)-1.3*(k)$);

  \draw[C] (L) -- (domain1);
  \draw[C1] (domain1) -- (U1);
  \draw[image] (U1) -- (D1);
  \draw[image1] (D1)-- (ker1);
  \draw (ker1) -- (L0);
  
  \fill (L) circle (\circlewidth);
  \fill (domain1) circle (\circlewidth);
  \fill (U1) circle (\circlewidth);
  \fill (D1) circle (\circlewidth);
  \fill (ker1) circle (\circlewidth);
  \fill (L0) circle (\circlewidth);

  \node at (L) [left] {$L$};
  \node at (domain1) [left] {$L^\phi \cong \img \imath_\phi$};
  \node at (U1) [left] {$L^{\phi \psi} \cong \img (\widetilde{\imath}\, \imath_\phi$)};
  \node at (ker1) [left] {$\ker \phi$};
  \node at (L0) [left] {$0_L$};
  
  \draw[map,-stealth'] ($(domain1)+0.1*(j)$) -- ($(Z1_phi)-0.1*(j)$);
  \draw[map,-stealth'] ($(U1)+0.1*(j)$) -- ($(int_Z2_phi)-0.1*(j)$);
  \draw[map,-stealth'] ($(D1)+0.1*(j)$) -- ($(B1+int_B2_phi)-0.1*(j)$);
  \draw[map,-stealth'] ($(ker1)+0.1*(j)$) -- ($(B1_phi)-0.1*(j)$);
  
  \node at ($0.5*(ker1)+0.5*(B1+int_B2_phi)-0.1*(g)$) {\color{gray} \footnotesize $\overunderline{\phi}$};
  \node at ($0.5*(D1)+0.5*(int_Z2_phi)-0.1*(f)$) {\color{gray} \footnotesize $\overunderline{\phi}$};
  \node at ($0.5*(U1)+0.5*(Z1_phi)-0.2*(b)$) {\color{gray} \tiny $\overunderline{\phi}$};
  
  \draw[map,-doublestealth] ($(M)-0.1*(j)$) -- ($(M_phi)+0.1*(j)$);
  \draw[map,-doublestealth] ($(B1)-0.1*(j)$) -- ($(B1_phi)+0.1*(j)$);
  
\end{tikzpicture}
\end{center}
}

The associativity of the composition is a tedious but elementary exercise.
This composition turns $\Gen(\A)$ into a category with the same class of objects as $\A$ and with the honest morphism
\[
  [1_M,1_M,1_M]
\]
(induced by $1_M$) as the identity of $M$.

The composition of two generalized morphisms $\phi$ and $\psi$ potentially ``shrinks'' the domain (compared to that of $\phi$) and codomain (compared to that of $\psi$). More precisely:
\begin{rmrk} \label{rmrk:domains_of_composition}
  The cokernel of the domain $\imath_{\phi \psi} = \widetilde{\imath}\, \imath_\phi$ of the composition is an extension of $\coker \imath_\phi$ by the subobject $\img \imath_\phi / \img(\widetilde{\imath}\, \imath_\phi)$, where the latter is naturally isomorphic to a subfactor of $\coker \imath_\psi$.
  Dually, the kernel of the codomain $\jmath_{\phi \psi} = \jmath_\psi \widetilde{\jmath}$ is an extension of $\ker \jmath_\psi$ by the factor object $\ker( \jmath_\psi \widetilde{\jmath}) / \ker \jmath_\psi$, where the latter is naturally isomorphic to a subfactor of $\ker \jmath_\phi$.
  These two claims can, for example, be easily read off from the above \nameft{Hasse} diagram of the composition.
\end{rmrk}

\subsection{The category \texorpdfstring{$\Gen(\A)$}{G(A)} is enriched over commutative inverse monoids}

Now, we introduce the natural (additively written) commutative inverse monoid structure on the $\Hom$-sets of $\Gen(\A)$ (cf.~Appendix~\ref{sec:isg}).
In contrast to the addition operation of generalized morphisms, which is somewhat technical since domains and codomains need to be unified, the additive inverse\footnote{in the sense of a commutative inverse semigroup} of a generalized morphism $\psi=[\imath_\psi,\overunderline{\psi},\jmath_\psi]$ is simply
\[
  -\psi := [\imath_\psi,-\overunderline{\psi},\jmath_\psi] \mbox{.}
\]

We now come to the definition of the addition.
For convenience we allow writing $\psi=[\imath_\psi,\overunderline{\psi},\jmath_\psi]$ where $\overunderline{\psi}:L \to N/K$ with $L \geq \img \imath_\psi$ and $K \leq \ker \jmath_\psi$.
Then we automatically replace the morphism $\overunderline{\psi}$ by its pre-composition with the natural mono $\img\imath_\psi\hookrightarrow L$ and its post-composition with the natural epi $N/K \twoheadrightarrow N/\ker \imath_\psi$.

\begin{defn}
  The \textbf{common restrictions} of two generalized morphisms $\beta=[\imath_\beta,\overunderline{\beta},\jmath_\beta]$ and $\gamma=[\imath_\gamma,\overunderline{\gamma},\jmath_\gamma]$ with the \emph{same source} is defined as the generalized morphisms
  \[
    \widetilde{\beta}:=[\kappa,\overunderline{\beta},\jmath_\beta]
    \mbox{ and }
    \widetilde{\gamma}:=[\kappa,\overunderline{\gamma},\jmath_\gamma] \mbox{,}
  \]
  where $\kappa$ is defined using the pullback below as $\kappa:=\alpha_\beta\imath_\beta=\alpha_\gamma\imath_\gamma$.

\begin{center}
  \begin{tikzpicture}[row sep=0.64cm,C/.style={color=red,line width=1pt}]
  \matrix (A) [matrix of nodes]
  {
  $\operatorname{pullback}(\imath_\beta,\imath_\gamma)$ &[1cm] $\source \imath_\gamma$ \\
  $\source \imath_\beta$ & $\source \gamma$ \\
  };
  \draw[-stealth'] (A-1-1) -- node[above] {$\alpha_\gamma$} (A-1-2);
  \draw[-stealth'] (A-1-2) -- node[right] {$\imath_\gamma$} (A-2-2);
  \draw[-stealth'] (A-1-1) -- node[left] {$\alpha_\beta$} (A-2-1);
  \draw[-stealth'] (A-2-1) -- node[below] {$\imath_\beta$} (A-2-2);
  \draw[-stealth'] (A-1-1) -- node[above right=-0.3em] {$\kappa$} (A-2-2);
  \end{tikzpicture}
\quad\quad\quad\quad
  \begin{tikzpicture}[row sep=0.5cm,C/.style={color=red,line width=1pt}]
  \matrix (A) [matrix of nodes]
  {
  $\target \gamma$ &[1cm] $\target\jmath_\gamma$ \\
  $\target\jmath_\beta$ & $\operatorname{pushout}(\jmath_\beta,\jmath_\gamma)$ \\
  };
  \draw[-stealth'] (A-1-1) -- node[above] {$\jmath_\gamma$} (A-1-2);
  \draw[-stealth'] (A-1-2) -- node[right] {$\alpha_\gamma$} (A-2-2);
  \draw[-stealth'] (A-1-1) -- node[left] {$\jmath_\beta$} (A-2-1);
  \draw[-stealth'] (A-2-1) -- node[below] {$\alpha_\beta$} (A-2-2);
  \draw[-stealth'] (A-1-1) -- node[above right=-0.3em] {$\lambda$} (A-2-2);
  \end{tikzpicture}
\end{center}
  
  The \textbf{common coarsenings} of two generalized morphisms $\beta=[\imath_\beta,\overunderline{\beta},\jmath_\beta]$ and $\gamma=[\imath_\gamma,\overunderline{\gamma},\jmath_\gamma]$ with the \emph{same target} is defined as the generalized morphisms
  \[
    \undertilde{\beta}:=[\imath_\beta,\overunderline{\beta},\lambda]
    \mbox{ and }
    \undertilde{\gamma}:=[\imath_\gamma,\overunderline{\gamma},\lambda] \mbox{,}
  \]
  where $\lambda$ is defined using the pushout above as $\lambda:=\jmath_\beta\alpha_\beta=\jmath_\gamma\alpha_\gamma$.
  
  The \textbf{common adaptations} of two parallel generalized morphisms (i.e., with the same source \emph{and} the same target) $\beta=[\imath_\beta,\overunderline{\beta},\jmath_\beta]$ and $\gamma=[\imath_\gamma,\overunderline{\gamma},\jmath_\gamma]$ is their common restrictions followed by their common coarsenings
  \[
    \widetilde{\undertilde{\beta}} := [\kappa,\overunderline{\beta},\lambda]
    \mbox{ and }
    \widetilde{\undertilde{\gamma}} := [\kappa,\overunderline{\gamma},\lambda] \mbox{.}
  \]
\end{defn}

\begin{defn} \label{defn:addition_subtraction_generalized}
Let $\phi=[\imath_\phi,\overunderline{\phi},\jmath_\phi],\psi=[\imath_\psi,\overunderline{\psi},\jmath_\psi]:M\to N$ be two generalized morphisms.
The sum and difference $\phi\pm \psi$ are defined by $\phi\pm\psi:=[\kappa,\overunderline{\phi}\pm\overunderline{\psi},\lambda]$, where $\widetilde{\undertilde{\phi}}=[\kappa,\overunderline{\phi},\lambda]$ and $\widetilde{\undertilde{\psi}}=[\kappa,\overunderline{\psi},\lambda]$ are the common adaptations of $\phi$ and $\psi$.
\end{defn}

Obviously, $\Hom_{\Gen(\A)}(M,N)$ with this addition and the above additive inversion is a commutative inverse semigroup with
\[
  0(\psi) := \psi - \psi = - \psi + \psi = [\imath_\psi,0_{M^\psi N_\psi},\jmath_\psi]
\]
the idempotent associated to $\psi$.
The bounded meet-semilattice $E(\Hom_{\Gen(\A)}(M,N))$ of idempotents is modular with a terminal object and consists of all generalized morphisms of the form $0(\psi)$.
In fact, $\Hom_{\Gen(\A)}(M,N)$ is a commutative inverse monoid with the honest morphism
\[
  [1_M, 0_{MN}, 1_N]
\]
(induced by $0_{MN} \in \Hom_\A(M,N)$) as the additive zero.

The category $\Gen(\A)$ is enriched over commutative inverse monoids equipped with the nonclosed symmetric tensor product of Definition~\ref{defn:tensor_cim}.
Note that $\Gen(\A)$ is in general \emph{not} enriched over commutative monoids with their standard tensor product:
The additive zeros are in general not absorbing, i.e., not preserved by the composition as the domain of $\psi 0$ is that of $\psi$ and the codomain of $0 \phi$ is that of $\phi$.

Analogous to Remark~\ref{rmrk:domains_of_composition} we observe that:
\begin{rmrk} \label{rmrk:domains_of_addition}
  The cokernel of the domain of $\phi \pm \psi$ is an extension of $\coker \imath_\psi$ by a factor object of $\coker \imath_\phi$, and vice versa.
  Dually, the kernel of the codomain of $\phi \pm \psi$ is an extension of $\ker \jmath_\psi$ by a subobject of $\ker \jmath_\phi$, and vice versa.
\end{rmrk}

\subsection{The computability of \texorpdfstring{$\Gen(\A)$}{G(A)}}

\begin{theorem} \label{thm:GA}
  If $\A$ is a constructively \nameft{Abel}ian category then $\Gen(\A)$ is constructively a category enriched over commutative inverse monoids.
\end{theorem}
\begin{proof}
\mbox{} \\
$\Gen(\A)$ is a \textbf{category}:
\begin{enumerate}
  \item\label{Gen_IdentityMorphism}
    The identity morphism $[1_M,1_M,1_M]$ of $M$ in $\Gen(\A)$ is constructible since the identity morphism $1_M$ in $\A$ is constructible.
  \item\label{Gen_Compose}
    All operations involved in Definition~\ref{defn:composition_generalized} of the composition are computable in the constructive \nameft{Abel}ian category $\A$.
  \item\label{Gen_HomSetMembership}
    A triple $(\imath_\psi,\overunderline{\psi},\jmath_\psi)$ is a representative of a generalized morphism of $\A$ iff $\imath_\psi,\overunderline{\psi},\jmath_\psi$ are $\A$-morphisms, the sources of $\imath_\psi,\overunderline{\psi}$ coincide, the targets of $\overunderline{\psi},\jmath_\psi$ coincide, and $\imath_\psi$ is an $\A$-mono (which can be verified by the vanishing of its kernel) and $\jmath_\psi$ is an $\A$-epi (which can be verified by the vanishing of its cokernel).
  \item\label{Gen_HomSetEquality}
    The equality of two generalized morphisms represented by two triples is decidable, as the conditions describing the equivalence of two representing triples are decidable in the constructively \nameft{Abe}ian category $\A$.
\suspend{enumerate}
$\Gen(\A)$ is a category \textbf{enriched over commutative inverse monoids}:
\resume{enumerate}
  \item[(e,f)]\label{Gen_AddMat} \label{Gen_SubMat}
    All operations involved in Definition~\ref{defn:addition_subtraction_generalized} of the addition and subtraction are computable in the constructively \nameft{Abel}ian category $\A$.
  \item[(g)] \label{Gen_ZeroMat}
  The zero morphism $[1_M,0_{MN},1_N]$ from $M \to N$ in $\Gen(\A)$ is constructible since all three morphisms $1_M,0_{MN},1_N$ are constructible in $\A$.
    \qedhere
\end{enumerate}
\end{proof}

In \cite{BaSF} the enrichment over \emph{commutative} inverse monoids plays no role, whereas other properties which we omitted here become more relevant.

\subsection{The subcategory \texorpdfstring{$G_\C(\A)$}{G\_C(A)} of \nameft{Gabriel} morphisms}

From now on $\C$ denotes a thick subcategory of the \nameft{Abel}ian category $\A$, i.e., a non-empty full subcategory of $\A$ closed under passing to sub- and factor objects and forming extensions.

\begin{defn} \label{defn:Gabriel_morphism}
  A \textbf{\nameft{Gabriel} morphism (of $\A$ with respect to $\C$)} is a generalized morphism $[\imath_\psi,\overunderline{\psi},\jmath_\psi]:M \to N$ with $\coker \imath_\psi \in \C$ and $\ker \jmath_\psi \in \C$.
\end{defn}

Denote by $\Gen_\C(\A)$ the subclass of \nameft{Gabriel} morphisms w.r.t. $\C$ in $\Gen(\A)$.
Remark~\ref{rmrk:domains_of_composition} states that $\Gen_\C(\A)$ is a subcategory of $\Gen(\A)$ (which is wide, i.e., with the same class of objects), while Remark~\ref{rmrk:domains_of_addition} ensures that $\Gen_\C(\A)$ inherits the enrichment of $\Gen(\A)$ over the monoidal category of commutative inverse monoids.
Note that $\Gen(\A) = \Gen_\A(A)$.

For the notions used in following proposition and its proof we refer the reader to Appendix~\ref{sec:isg}.
\begin{prop} \label{prop:Zeroid_of_Gabriel}
  The zeroid of the commutative inverse monoid $\Hom_{\Gen_\C(\A)}(M,N)$ consists of all \nameft{Gabriel} morphism $\psi=[\imath_\psi,\overunderline{\psi},\jmath_\psi]:M \to N$ with $\img \overunderline{\phi} \in \C$.
\end{prop}
\begin{proof}
  Consider such a $\psi$ and denote by $\pi_\psi$ the cokernel epi of $\overunderline{\psi}$.
  Define $\phi := [\imath_\psi, \overunderline{\psi} \pi_\psi, \jmath_\psi \pi_\psi]$.
  Then $\ker \left(\jmath_\psi \pi_\psi\right) \in \C$ and $\phi = \phi + \psi$, i.e., $\psi$ belongs to the zeroid by Lemma~\ref{lemm:zeroid_characterization}.
  Conversely, let $\phi= \phi+\psi$.
  By replacing $\phi$ with $0(\phi)$ we may assume $\phi$ an idempotent.
  And by replacing the idempotent $\phi$ with $\phi+0(\psi)$ we may assume that $\phi \leq \psi$, i.e., that the (co)domain of $\psi$ (co)dominates\footnote{A mono $\kappa$ is said to dominate the mono $\lambda$ if there exists a (necessarily unique) lift $\liftofCalongB{\kappa}{\lambda}$ of $\lambda$ along $\kappa$, i.e., $(\liftofCalongB{\kappa}{\lambda}) \kappa = \lambda$. The notion of codomination is the dual one.} that of the \emph{idempotent} $\phi$.
  Since $\phi=\phi+\psi$ it now follows that $\img \overunderline{\psi} \in \C$.
\end{proof}

The bounded meet-semilattice $E(\Hom_{\Gen_\C(\A)}(M,N))$ of idempotents is a modular sublattice of $E(\Hom_{\Gen(\A)}(M,N))$, but in general without a terminal object.

\begin{defn}
  Define the \textbf{zeroid} $Z(\Gen_\C(\A))$ of the category $\Gen_\C(\A)$ as the disjoint union (over all $M,N \in \A$) of the zeroids of the commutative inverse monoids $\Hom_{\Gen_\C(\A)}(M,N)$.
\end{defn}

\begin{rmrk} \label{rmrk:GenCA_and_Zeroid}
  Let $\A$ be a constructively \nameft{Abel}ian category and $\C \subset \A $ a thick subcategory for which we can decide\footnote{By ``decide'' we always mean ``algorithmically decide''.} the membership problem.
  It is evident from Definition~\ref{defn:Gabriel_morphism} that the membership in the wide subcategory $\Gen_\C(\A) \subset \Gen(\A)$ becomes decidable as well.
  Hence, by Theorem~\ref{thm:GA}, $\Gen_\C(\A)$ is constructively a category enriched over commutative inverse monoids.
  
  Furthermore, the zeroid $Z(\Gen_\C(\A))$ is a two-sided ideal in $\Gen_\C(\A)$ (see Remark~\ref{rmrk:zeroid_semirings} for the corresponding statement in the context of semirings) for which membership is decidable by Proposition~\ref{prop:Zeroid_of_Gabriel}.
\end{rmrk}

\section{The computability of \nameft{Serre} quotient categories} \label{sec:computability_quotient}

In this section we explicitly verify that \nameft{Serre} quotient categories $\A/\C$ of an \nameft{Abel}ian category $\A$ modulo a thick subcategory $\C$ are constructively \nameft{Abel}ian once $\A$ is constructively \nameft{Abel}ian and the membership in $\C$ is decidable.
This proof uses the \nameft{Gabriel} morphisms from last section and underlies our computer implementation shown in Appendix~\ref{sec:ex}.

\subsection{Preliminaries} \label{subsec:Serre}

The \textbf{\nameft{Serre} quotient category} $\A/\C$ has the same object class as $\A$ and $\Hom$-groups defined by
\[
  \Hom_{\A/\C}(M,N) = \varinjlim_{\substack{M' \hookrightarrow M, N' \hookrightarrow N \\ M/M', N' \in \C}} \Hom_\A(M',N/N')\mbox{.}
\]
The \textbf{canonical functor} $\QQ:\A \to \A/\C$ is the identity on objects and maps a morphism $\phi \in \Hom_\A(M,N)$ to its image in $\Hom_{\A/\C}(M,N)$ under the maps
\[
  \Hom(M' \hookrightarrow M, N \twoheadrightarrow N/N'):\Hom_\A(M,N) \to \Hom_\A(M',N/N') \mbox{.}
\]
Again, $\A/\C$ is an \nameft{Abel}ian category and the canonical functor $\QQ: \A \to \A/\C$ is \emph{exact}.

Any morphism $M\xrightarrow{\psi}N$ in $\A/\C$ thus stems from a morphism in the category $\A$ with a subobject $M'$ of $M$ as source and a factor object $N/N'$ of $N$ as target.
The category $\Gen_\C(\A)$ now incorporates all morphisms in the defining direct system of the colimit \nameft{Abel}ian group $\Hom_{\A/\C}(M,N)$ into the single commutative inverse monoid $\Hom_{\Gen_\C(\A)}(M,N)$.
Applying Corollary~\ref{coro:zeroid_is_kernel_of_colimit} we can thus replace the colimit description of $\Hom_{\A/\C}(M,N)$ by factoring out the smallest group congruence in $\Hom_{\Gen_\C(\A)}(M,N)$.
More precisely:
\begin{coro} \label{coro:Serre_is_Gabriel_modulo_zeroid}
  The canonical homomorphism $\Hom_{\Gen_{\C}(\A)}(M,N) \to \Hom_{\A/\C}(M,N)$ of semigroups induces an isomorphism of \nameft{Abel}ian groups
  \[
    \Hom_{\A/\C}(M,N) \cong \Hom_{\Gen_{\C}(\A)}(M,N)/Z_{M,N} \mbox{.}
  \]
  where $Z_{M,N}$ is the zeroid of the commutative inverse monoid $\Hom_{\Gen_{\C}(\A)}(M,N)$.
  Furthermore, the \nameft{Serre} quotient $\A/\C$ is equivalent to the factor category $\Gen_\C(\A) / Z(\Gen_\C(\A))$.\footnote{
    Such factor categories appear under the name ``quotient category'' in \cite{RadicalOfACategory}.
    Another example: The homotopy category $\operatorname{K}^?(\A)$ of an \nameft{Abel}ian $\A$ is defined as the factor category of chain complexes $\operatorname{Ch}^?(\A)$ by the two-sided ideal of zero-homotopic chain morphisms. There the outcome is triangulated and in general not \nameft{Abel}ian.
  }
\end{coro}

\subsection{\texorpdfstring{$\A/\C$}{A/C} is constructively \nameft{Abel}ian} \label{subs:proof}

\nameft{Gabriel} proves in \cite[Proposition~III.1.1]{Gab_thesis} that $\A/\C$ is an \nameft{Abel}ian category.
In his proof he assumes at the beginning of each construction (e.g., of kernels and cokernels) that the involved $\A/\C$-morphism $M \to N$ is expressible as an $\A$-morphism $M \to N$ (which is the same as an honest \nameft{Gabriel} morphism).
By this he allows adapting the models for source and target of such morphisms a priori, which is justifiable by the colimit.
This assumption strongly simplifies his proof as it completely hides the colimit process behind a ``without loss of generality'' statement.
Explicitly keeping track of the colimit process by advancing along the corresponding direct system renders \nameft{Gabriel} proof constructive.
This is precisely achieved by the \nameft{Gabriel} morphisms which naturally evolved as a data structure for morphisms in \nameft{Serre} quotient categories suitable for a computer implementation.

\begin{proof}[Proof of Theorem~\ref{thm:gabriel.computable}]

Let $L$, $M$, and $N$ be objects of $\A$ which we also consider as objects of $\A/\C$.
We go through all disjunctions and all existential quantifiers listed in Appendix~\ref{sec:ComputableAbelian} and show how to turn them into constructive ones using \nameft{Gabriel} morphisms in $\A$ with respect to $\C$ as a model for morphisms in $\A/\C$.

\noindent
$\A/\C$ is a category \textbf{enriched over commutative inverse monoids\footnote{Cf.~Definition~\ref{defn:tensor_cim}.}}:
\begin{enumerate}
  \item[(a-g)]\label{Serre_Enriched}
    Remark~\ref{rmrk:GenCA_and_Zeroid} states that $\Gen_\C(\A)$ is a constructively a category enriched over commutative inverse monoids and that the membership in its zeroid ideal $Z(\Gen_\C(\A))$ is decidable.
    All constructions now follow from Corollary~\ref{coro:Serre_is_Gabriel_modulo_zeroid} where two morphisms in $\A/\C$ are equal if and only if the difference of the representing Gabriel morphisms lies in $Z(\Gen_\C(\A))$.
  \setcounter{enumi}{\value{enumi}+7}
  \setcounter{saveenum}{\value{enumi}}
\suspend{enumerate}
$\A/\C$ is a category \textbf{with zero}:
\resume{enumerate}
  \item
    A zero object in $\A/\C$ can be modeled by a zero object in $\A \subset \Gen_\C(\A)$.
\suspend{enumerate}
$\A/\C$ is a \textbf{pre-additive} category:
\resume{enumerate}
  \item\label{Serre_AbelianGroup}
  The definition of the zeroid and Corollary~\ref{coro:Serre_is_Gabriel_modulo_zeroid} imply that $\Hom_{\A/\C}(M,N)$ with the above $(+,-,0_{MN})$ is an \nameft{Abel}ian group.
\suspend{enumerate}
$\A/\C$ is an \textbf{additive} category:
\resume{enumerate}
  \item\label{Serre_DiagMatProd}
    The $\A/\C$-direct sum $N \oplus L$ is modeled by the $\A$-direct sum $N\oplus L$ and the $\A/\C$-projections are modeled by the honest \nameft{Gabriel} morphisms induced by the $\A$-projections.
  \item\label{Serre_UnionOfColumns}
    Let $\phi=[\imath_\phi,\overunderline{\phi},\jmath_\phi]:M \to N,\, \psi=[\imath_\psi,\overunderline{\psi},\jmath_\psi]: M \to L$ be two \nameft{Gabriel} morphisms in $\A$ with respect to $\C$.
    The product morphism $\{\phi,\psi\}:M\to N\oplus L$ is modeled by passing to the common restrictions $\widetilde{\phi}=[\kappa,\overunderline{\phi},\jmath_\phi]$ and $\widetilde{\psi}=[\kappa,\overunderline{\psi},\jmath_\psi]$ of $\phi$ and $\psi$ and taking $\{\phi,\psi\}:=[\kappa,\{\overunderline{\phi},\overunderline{\psi} \},\jmath_\phi\oplus\jmath_\psi]$.
\suspend{enumerate}
$\A/\C$ is a \textbf{pre-\nameft{Abel}ian} category:
\resume{enumerate}
  \item\label{Serre_SyzygiesGenerators}
    Let $\phi=[\imath_\phi,\overunderline{\phi},\jmath_\phi]:M \to N$.
    The kernel $\ker \phi \stackrel{\kappa}{\hookrightarrow} M$ is modeled by the honest \nameft{Gabriel} morphism induced by the composition $\ker \overunderline{\phi} \hookrightarrow \source \imath_\phi\stackrel{\imath_\phi}{\hookrightarrow}M$.
  \item\label{Serre_DecideZeroEffectively} Computing the lift along the kernel mono we use Corollary~\ref{cor:lift_of_gabriel_morphisms} below.
  \item\label{Serre_UnionOfRows_IdentityMatrix}
  Let $\phi=[\imath_\phi,\overunderline{\phi},\jmath_\phi]:M \to N$.
    The cokernel $N\twoheadrightarrow \coker \phi$ is modeled by the honest \nameft{Gabriel} morphism induced by the composition $N\stackrel{\jmath_\phi}{\twoheadrightarrow}\target \jmath_\phi\twoheadrightarrow \coker \overunderline{\phi}$.
  \item\label{Serre_colift} Computing the colift is the dual of Lemma~\ref{lemm:lift_of_gabriel_morphisms} and its Corollary~\ref{cor:lift_of_gabriel_morphisms} below.
\suspend{enumerate}
$\A/\C$ is a \textbf{\nameft{Abel}ian} category:
\resume{enumerate}
  \item\label{Serre_MonoLift} Computing the lift along a mono is the statement of Corollary~\ref{cor:lift_of_gabriel_morphisms} below.
  \item\label{Serre_EpiColift} Computing the colift along an epi is again the dual statement.\qedhere
\end{enumerate}
\end{proof}

For the sake of an efficient implementation, we recommend giving direct constructions for the natural embeddings into a coproduct and for the co-pairing morphism.
\begin{enumerate}[\quad\,\,(a')]
\setcounter{enumi}{\value{saveenum}}
\item\label{Serre_DiagMatCoprod}
  The embeddings are induced by the $\A$-embeddings.
\item\label{Serre_UnionOfRows}
  Let $\phi=[\imath_\phi,\overunderline{\phi},\jmath_\phi]:N \to M,\, \psi=[\imath_\psi,\overunderline{\psi},\jmath_\psi]: L \to M$ be two \nameft{Gabriel} morphisms in $\A$ with respect to $\C$.
  The co-pairing morphism $\langle\phi,\psi\rangle:N\oplus L \to M$ is modeled by passing to the common coarsenings $\undertilde{\phi}=[\imath_\phi,\overunderline{\phi},\lambda]$ and $\undertilde{\psi}=[\imath_\psi,\overunderline{\psi},\lambda]$ of $\phi$ and $\psi$ and taking $\langle\phi,\psi\rangle:=[\imath_\phi\oplus\imath_\psi,\langle\overunderline{\phi},\overunderline{\psi} \rangle,\lambda]$.
\end{enumerate}

\smallskip

\begin{wrapfigure}[6]{r}{6.8cm}
\centering
\vskip -1.2em
\begin{tikzpicture}
    \node (M)                             {$M_\gamma=M_\beta$};
    \node (L)          [left=1.5 cm of M]   {$L^\gamma$};
    \node (K)          [right=1.5 cm of M]  {$K^\beta$};
    \node (gammaminus1)[below=of L]       {$\overline{L}$};
    \node (modkerbeta) [below=of K]       {$\underline{K}$};
    \draw[-stealth'] (L)                    -- node[above]         {$\overunderline{\gamma}$}               (M);
    \draw[-stealth'] (gammaminus1)          -- node[below right=0.2cm]   {$\overunderline{\gamma}$} (M);
    \draw[-stealth'] (K)                    -- node[above]         {$\overunderline{\beta}$}                (M);
    \draw[left hook-stealth'] (gammaminus1)           -- node[left]          {$\imath$}                 (L);
    \draw[-doublestealth] (K)                   -- node[right]         {$\jmath$} (modkerbeta);
    \draw[-stealth',dotted] (gammaminus1)   -- node[below]         {$\liftofCalongB{\underline{\beta}}{\overline{\gamma}}$}              (modkerbeta);
    \draw[left hook-stealth'] (modkerbeta)           -- node[below left=0.2cm]    {$\overunderline{\beta}$}                 (M);
\end{tikzpicture}
\end{wrapfigure}

\subsection{The Lifting Lemma}
We need some language to describe lifts in the context of \nameft{Gabriel} morphisms.
Let $\beta:=[\imath_\beta,\overunderline{\beta},\lambda]:K\to M$ and $\gamma:=[\imath_\gamma,\overunderline{\gamma},\lambda]:L\to M$ be two \nameft{Gabriel} morphisms; we may assume by a common coarsening that they have equal codomain\footnote{In particular, the targets of $\overunderline{\gamma}$ and $\overunderline{\beta}$ coincide: $M_\beta=M_\gamma=\target \lambda$.} $\lambda$.
We define the subobject $\overline{L} := \overunderline{\gamma}^{-1}(\img \overunderline{\beta}) \leq L^\gamma$ and call it the \textbf{$(\beta,\gamma)$-adapted source of $\gamma$}.
We denote its embedding by $\imath: \overline{L} \hookrightarrow L^\gamma$, and, by abuse of notation, the restriction of $\overunderline{\gamma}$ to $\overline{L}$ by $\overunderline{\gamma}:=\imath\,\overunderline{\gamma}: \overline{L} \to M_\gamma$.
Set $\underline{K} := K^\beta / \ker \overunderline{\beta}$ and call it the \textbf{$(\beta,\gamma)$-adapted source of $\beta$}.
We denote the natural epi by $\jmath:K^\beta\twoheadrightarrow \underline{K}$ and, again by abuse of notation, the induced morphism $\underline{K}\to M_\beta$ by $\overunderline{\beta}$.
By definition of $\overline{L}$ and $\underline{K}$ there exists a unique lift
\[
  \liftofCalongB{\underline{\beta}}{\overline{\gamma}}: \overline{L}\to \underline{K}
\]
of $\overunderline{\gamma}: \overline{L} \to M_\gamma$ along the mono $\overunderline{\beta}:\underline{K}\to M_\beta=M_\gamma$.
Its constructibility is guaranteed by Axiom \eqref{Ab_MonoLift} in Appendix~\ref{sec:ComputableAbelian} of the constructively \nameft{Abel}ian category $\A$.

\begin{lemm}[Lifting Lemma for \nameft{Gabriel} morphisms]\label{lemm:lift_of_gabriel_morphisms}
  Using the above notation let $\gamma:=[\imath_\gamma,\overunderline{\gamma},\lambda]:L\to M$ and $\beta:=[\imath_\beta,\overunderline{\beta},\lambda]:K\to M$ be two \nameft{Gabriel} morphisms with equal codomain $\lambda$.
  Furthermore, let $\ker\overunderline{\beta}$ and $(\img\overunderline{\beta} + \img\overunderline{\gamma}) / \img\overunderline{\beta}$ be in $\C$.
  Then $\liftofCalongB{\beta}{\gamma}:=\left[\imath\, \imath_\gamma, \liftofCalongB{\underline{\beta}}{\overline{\gamma}}, \jmath_\beta\,\jmath\right]$ is a \nameft{Gabriel} morphism $L \to K$ and the common adaptations of $(\liftofCalongB{\beta}{\gamma})\beta$ and $\gamma$ coincide.
\end{lemm}
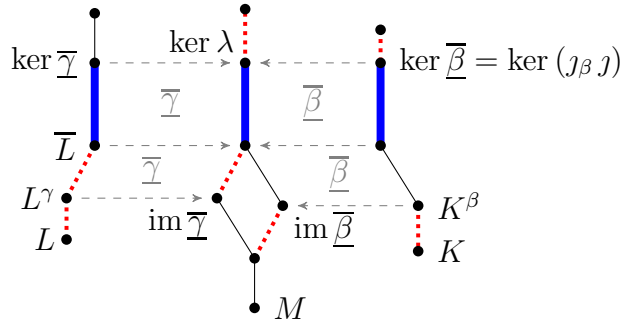
\begin{center}
\begin{tikzpicture}[C/.style={color=red,line width=1.5pt,dotted},
  image/.style={color=blue,line width=3},
  map/.style={color=gray,dashed}]
  \coordinate (a) at (0,-1.1);
  \coordinate (b) at ($(a)$);
  \coordinate (c) at (-0.375,-0.70);
  \coordinate (d) at (2,0);
  \coordinate (e) at (0.5,-0.8);
  \coordinate (f) at ($0.6*(a)$);
  \coordinate (g) at ($0.4*(a)$);
  \coordinate (h) at ($0.9*(d)$);
  \coordinate (i) at ($0.5*(a)$);
  \coordinate (j) at ($-0.65*(a)$);
  \coordinate (k) at ($0.55*(a)$);
  \coordinate (right) at (0.2,0);
  \coordinate (left) at ($-1*(right)$);
  
  \coordinate (0_L);
  \coordinate (ker_gamma) at ($(0_L)+0.6*(a)$);
  \coordinate (adapted_source) at ($(ker_gamma)+(b)$);
  \coordinate (im_i_gamma) at ($(adapted_source)+(c)$);
  \coordinate (L) at ($(im_i_gamma)+(i)$);
  \coordinate (ker_lambda) at ($(ker_gamma)+(d)$);
  \coordinate (0_M) at ($(ker_lambda)+(j)$);
  \coordinate (im_cap_im) at ($(ker_lambda)+(b)$);
  \coordinate (im_gamma) at ($(im_cap_im)+(c)$);
  \coordinate (im_beta) at ($(im_cap_im)+(e)$);
  \coordinate (im_plus_im) at ($(im_beta)+(c)$);
  \coordinate (M) at ($(im_plus_im)+(f)$);
  \coordinate (ker_beta) at ($(ker_lambda)+(h)$);
  \coordinate (0_K) at ($(ker_beta)-(g)$);
  \coordinate (K2) at ($(ker_beta)+(b)$);
  \coordinate (im_i_beta) at ($(K2)+(e)$);
  \coordinate (K) at ($(im_i_beta)+(k)$);
  
  
  \draw (0_L) -- (ker_gamma);
  \draw[image] (ker_gamma) -- (adapted_source);
  \draw[image] (ker_beta) -- (K2);
  \draw (K2) -- (im_i_beta);
  \draw[image] (ker_lambda) -- (im_cap_im);
  \draw (im_gamma) -- (im_plus_im) -- (M);
  \draw (im_cap_im) -- (im_beta);
  
  \draw[C] (adapted_source) -- (im_i_gamma);
  \draw[C] (0_K) -- (ker_beta);
  \draw[C] (im_cap_im) -- (im_gamma);
  \draw[C] (im_beta) -- (im_plus_im);
  \draw[C] (im_i_gamma) -- (L);
  \draw[C] (ker_lambda) -- (0_M);
  \draw[C] (im_i_beta) -- (K);
  
  \draw[map,-stealth'] ($(ker_gamma)+(right)$) -- ($(ker_lambda)+(left)$);
  \draw[map,-stealth'] ($(adapted_source)+(right)$) -- ($(im_cap_im)+(left)$);
  \draw[map,-stealth'] ($(im_i_gamma)+(right)$) -- ($(im_gamma)+(left)$);

  \node at ($0.5*(im_i_gamma)+0.5*(im_cap_im)+0.15*(c)$) {\color{gray} $\overunderline{\gamma}$};
  \node at ($0.5*(adapted_source)+0.5*(ker_lambda)+0.1*(b)$) {\color{gray} $\overunderline{\gamma}$};
  
  \draw[map,-stealth'] ($(ker_beta)+(left)$) -- ($(ker_lambda)+(right)$);
  \draw[map,-stealth'] ($(K2)+(left)$) -- ($(im_cap_im)+(right)$);
  \draw[map,-stealth'] ($(im_i_beta)+(left)$) -- ($(im_beta)+(right)$);
  
  \node at ($0.5*(im_i_beta)+0.5*(im_cap_im)+0.2*(e)$) {\color{gray} $\overunderline{\beta}$};
  \node at ($0.5*(K2)+0.5*(ker_lambda)+0.15*(b)$) {\color{gray} $\overunderline{\beta}$};
  
  \fill (0_L) circle (2pt);
  \fill (ker_gamma) circle (2pt) node[left=3pt] {$\ker\overunderline{\gamma}$};
  \fill (adapted_source) circle (2pt) node[left=3pt] {$\overline{L}$};
  \fill (im_i_gamma) circle (2pt) node[left] {$L^\gamma$};
  \fill (L) circle (2pt) node[left] {$L$};
  \fill (ker_lambda) circle (2pt) node[above left] {$\ker\lambda$};
  \fill (0_M) circle (2pt);
  \fill (im_cap_im) circle (2pt);
  \fill (im_gamma) circle (2pt) node[below left] {$\img\overunderline{\gamma}$};
  \fill (im_beta) circle (2pt) node[below right] {$\img\overunderline{\beta}$};
  \fill (im_plus_im) circle (2pt);
  \fill (M) circle (2pt) node[right=3pt] {$M$};
  \fill (ker_beta) circle (2pt) node[right=3pt] {$\ker\overunderline{\beta}=\ker\left(\jmath_\beta\,\jmath\right)$};
  \fill (0_K) circle (2pt);
  \fill (K2) circle (2pt);
  \fill (im_i_beta) circle (2pt) node[right=3pt] {$K^\beta$};
  \fill (K) circle (2pt) node[right=3pt] {$K$};
  
\end{tikzpicture}
\end{center}
\begin{proof}
The proof involves no further constructions.
The cokernel of $\imath\, \imath_\gamma$ lies in $\C$ as the extension of the two objects $L/L^\gamma \cong \coker \imath_\gamma$ and $L^\gamma/\overline{L} \cong (\img\overunderline{\beta} + \img\overunderline{\gamma}) / \img\overunderline{\beta}$, both lying in $\C$.
The kernel of $\jmath_\beta\,\jmath$ coincides with $\ker \overunderline{\beta}$ and is hence in $\C$.

The rest of the claim can be read off the \nameft{Hasse} diagram where we identify the objects $L^\gamma$ and $\overline{L}=\img\left(\imath\,\imath_\gamma\right)$ with their images in $L$, $K^\beta$ with its image in $K$, and $M_\beta=M_\gamma$ with $M/\ker\lambda$.
\end{proof}

\begin{coro} \label{cor:lift_of_gabriel_morphisms}
Let $[\beta]:K \to M$ be a mono  in $\A/\C$ and $[\gamma]: L \to M$ another morphism such that its composition with the cokernel epi of $[\beta]$ vanishes.
Without loss of generality we can assume that $[\beta]$ and $[\gamma]$ are represented by the \nameft{Gabriel} morphisms $\beta:=[\imath_\beta,\overunderline{\beta},\lambda]$ and $\gamma:=[\imath_\gamma,\overunderline{\gamma},\lambda]$ with equal codomain $\lambda$, respectively.
Then the unique lift of $[\gamma]$ along $[\beta]$ is represented by the \nameft{Gabriel} morphism $\liftofCalongB{\beta}{\gamma}:=\left[\imath\, \imath_\gamma, \liftofCalongB{\underline{\beta}}{\overline{\gamma}}, \jmath_\beta\,\jmath\right]: L \to K$.
\end{coro}
\begin{proof}
The two conditions on $[\beta]$ and $[\gamma]$ are nothing but the two condition $\ker\overunderline{\beta} \in \C$ and $(\img\overunderline{\beta} + \img\overunderline{\gamma}) / \img\overunderline{\beta} \in \C$ in the Lifting Lemma~\ref{lemm:lift_of_gabriel_morphisms}.
\end{proof}

\begin{rmrk} \label{rmrk:C_as_obstruction_to_lift}
The figure in the proof of Lemma~\ref{lemm:lift_of_gabriel_morphisms} shows that lifts automatically introduce non-full domains as soon as $\ker \overunderline{\beta}$ is non-zero (in $\C$) and introduce non-full codomains as soon as $(\img\overunderline{\beta} + \img\overunderline{\gamma}) / \img\overunderline{\beta}$ is non-zero (in $\C$).
Summing up, objects in $\C$ often arise as obstructions to the existence of certain (co)lifts in $\A$.
Computing modulo $\C$, i.e., computing in $\A/\C$ these obstructions vanish and such (co)lifts exist.
\end{rmrk}

\begin{exmp}
  Let $\A$ be the category of finitely generated \nameft{Abel}ian groups (or $\Z$-modules) and $\C$ the subcategory of finitely generated torsion \nameft{Abel}ian groups.
  Consider the two $\A$-morphisms $\beta:\Z \stackrel{4}{\hookrightarrow} \Z$ and $\gamma:\Z \stackrel{6}{\to} \Z$.
  The lift $\liftofCalongB{\beta}{\gamma}$ does not exist in $\A$ as the image of $\gamma$ is not contained in the image of $\beta$.
  But viewing $\beta$ and $\gamma$ as morphisms in $\A/\C$ represented by honest \nameft{Gabriel} morphisms in $\A$ w.r.t.\ $\C$, the unique lift exists:
  \[
    \liftofCalongB{\beta}{\gamma} = \left( \Z \stackrel{2}{\hookrightarrow} \Z, \Z \xrightarrow{3} \Z, \Z \stackrel{1}{\hookrightarrow} \Z \right) \mbox{.}
  \]
\end{exmp}

Any good computer implementation should immediately profit when domains or codomains are full, since then we get closer to compute in $\A$ rather than in $\A/\C$.
In the language of \nameft{Hasse} diagrams this means that a lot of dotted lines, which stand for subfactors in $\C$, disappear.
In our application to coherent sheaves this is indeed very often the case.

\section{Computability of categories of coherent sheaves} \label{sec:sheaves}

In this section we apply the framework of \nameft{Gabriel} morphisms and show that categories of coherent sheaves on projective schemes and smooth toric varieties are constructively \nameft{Abel}ian.
Therefore, we consider the computability of categories of finitely presented graded modules.

\subsection{The computability of certain (graded) module categories}

In \cite{BL} we showed, that categories of finitely presented modules are constructively \nameft{Abel}ian if the corresponding ring is computable.
In this subsection we formulate the same result in the graded context and describe certain computable graded rings needed for the description of sheaves.
We call a (unitial) commutative ring $R$ \textbf{computable} if there exists an algorithm to solve a linear systems over $R$, i.e., to find an (affine) generating set of all $X$ with $B=XA$ for given matrices $A$ and $B$ over $R$.

\begin{theorem}[{\cite[Theorem~3.4, §3.3,4]{BL}}]\label{thm:modules.computable}
  If $R$ is a computable ring then the category of finitely presented $R$-modules is constructively \nameft{Abel}ian.
\end{theorem}

One can analogously define the notion of a computable graded ring.
In the following let $D$ be a finitely generated \nameft{Abel}ian group.
We call a $D$-graded commutative ring $\grS$ \textbf{computable (as a $D$-graded ring)} if there exists an algorithm to find \emph{homogeneous} generating sets of affine spaces of solutions of linear systems $B=XA$ over $\grS$.
Similar to Theorem~\ref{thm:modules.computable} one shows the following corollary.

\begin{coro} \label{coro:graded.computability}
  If $\grS$ is a computable $D$-graded ring then the category $\A=\Sfpgrmod$ of finitely presented graded $\grS$-modules is constructively \nameft{Abel}ian.
\end{coro}

Now, we describe two classes of computable rings.

The first class are multigraded polynomial rings over a field, which we encounter in the toric setting.
Let $\grS = k[x_1,\ldots,x_n]$ be graded over a finitely generated \nameft{Abel}ian group $D$ satisfying
\begin{align} \label{eq:gr}
  &\mbox{$k$ is a computable field} \tag{*} \\
  &\mbox{and $\deg(a) = 0 \in D$ for all $a \in k^*$.} \nonumber
\end{align}

The $D$-graded ring $\grS = k[x_1,\ldots,x_n]$ satisfying \eqref{eq:gr} is computable by standard \nameft{Gröbner} bases methods (e.g.~\cite[§3.5,3.7]{AL}), which automatically respect the additional $D$-grading.

\begin{coro} \label{coro:Sfpgrmod_computable}
  The category of finitely presented graded modules over the graded ring $\grS=k[x_1,\ldots,x_n]$ satisfying \eqref{eq:gr} is constructively \nameft{Abel}ian.
\end{coro}

The second class of rings describes projective varieties over an affine scheme $\Spec B$.

\begin{defn}[{\cite[\S 4.3]{AL}}]\label{defn:coset.representatives}
  A ring $B$ is said to have \textbf{effective coset representatives} if for every ideal $I$ we can determine a set $T$ of coset representatives of $B/I$, such that for every $b \in B$ we can compute a unique $t \in T$ with $b+I=t +I$.
\end{defn}

Many rings have this property, e.g., fields and $\Z$.
Furthermore, if $B$ has effective coset representatives then its residue class rings and polynomial rings as well.
In the following we consider the $\Z$-graded polynomial ring $\grS=B[x_0,\ldots,x_n]$ satisfying
\begin{align} \label{eq:ecr}
  &\mbox{$B$ is a computable ring with effective coset representatives} \tag{**} \\
  &\mbox{and $\deg(x_0)=\ldots=\deg(x_n)=1$ and $\deg(b)=0$ for all $b\in B \setminus \{0\}$.} \nonumber
\end{align}

\begin{prop} \label{prop:projective.ring.computable}
  The $\Z$-graded ring $\grS=B[x_0,\ldots,x_n]$ satisfying \eqref{eq:ecr} is computable as a graded ring and has effective coset representatives.
\end{prop}  
\begin{proof}
  There are well-known \nameft{Gröbner} bases techniques for $\grS$.
  Theorem~4.3.3 in \cite{AL} shows that reduction yields unique elements and its proof shows that the coefficients of the representation of reduced elements are homogeneous.
  This also implies that the reduced elements are themselves homogeneous.
  Finally, the syzygies construction above \cite[Theorem~4.3.15]{AL} produces homogeneous syzygies.
\end{proof}

\begin{coro} \label{coro:SfpZgrmod_computable}
  The category of finitely presented graded modules over the graded ring $\grS=B[x_0,\ldots,x_n]$ satisfying \eqref{eq:ecr} is constructively \nameft{Abel}ian.
\end{coro}

\subsection{Coherent sheaves on projective spaces}

In this subsection we prove Theorem~\ref{thm:proj.computable} establishing the computability of the \nameft{Abel}ian category $\Coh \PP^n_B$ of coherent sheaves on the projective space $\PP^n_B$ (over $\Spec B$).
To this end let $\grS=B[x_0,\ldots,x_n]$ be a $\Z$-graded polynomial ring over $B$ with $\deg(x_0)=\ldots=\deg(x_n)=1$ and $\deg(B)=\{0\}$.
Denote by $\Sfpgrmod$ the category of finitely presented $\Z$-graded $\grS$-modules and by $\Sfpgrmod^0 \subset \Sfpgrmod$ its full subcategory of quasi-zero modules, i.e., those with $\grM_d = 0 $ for $d \gg 0$.

\begin{theorem}[{\cite[Corollary~4.2]{BL_SerreQuotients}}] \label{thm:coh.equiv.local.modules}
  Let $B$ be a \nameft{Noether}ian ring.
  The exact and essentially surjective sheafification functor $\Sfpgrmod \to \Coh \PP^n_B, M \mapsto \widetilde{\grM}$ induces an equivalence
  \[
    \Sfpgrmod / \Sfpgrmod^0 \xrightarrow{\sim} \Coh \PP^n_B
  \]
  of categories, where $\Sfpgrmod^0$ coincides with the kernel of the sheafification functor.
\end{theorem}

\begin{lemm} \label{lemm:decide_quasi_zero}
  Let the ring $B$ have effective coset representatives, i.e., $\grS=B[x_0,\ldots,x_n]$ satisfies \eqref{eq:ecr}.
  Then, one can decide whether a module is contained in $\Sfpgrmod^0$.
\end{lemm}
\begin{proof}
  Restricting $\grM \in \Sfpgrmod$ to the $i$-th open standard affine chart yields the module $(\grM_{x_i})_0$ over the polynomial ring $(\grS_{x_i})_0=B\left[\frac{x_0}{x_i},\ldots,\frac{x_n}{x_i}\right]$ (with the same presentation matrix as $M$).
  The sheafification $\widetilde{\grM} =0$ if and only if the restricted modules vanish for all $0\le i\le n$.
  (In case $B$ is a computable field then $\grM\in\Sfpgrmod^0$ if and only if the \nameft{Hilbert} polynomial vanishes.)
\end{proof}

Now we can prove Theorem~\ref{thm:proj.computable}, i.e., that the category $\Coh \PP^n_B$ is constructively \nameft{Abel}ian if $B$ is a computable commutative ring with effective coset representatives.

\begin{proof}[Proof of Theorem~\ref{thm:proj.computable}]
  First, the category of finitely presented graded $\grS$-modules is constructively \nameft{Abel}ian by Corollary~\ref{coro:SfpZgrmod_computable}.
  Second, we can decide whether a module is contained in $\Sfpgrmod^0$ by Lemma~\ref{lemm:decide_quasi_zero}.
  Thus, by Theorem~\ref{thm:gabriel.computable} the category $\Sfpgrmod / \Sfpgrmod^0$ is constructively \nameft{Abel}ian.
  The statement follows from the equivalence of categories in Theorem~\ref{thm:coh.equiv.local.modules}.
\end{proof}

The category of coherent sheaves over a closed subscheme $X$ of $\PP^n_B$ is also constructively \nameft{Abel}ian:
instead of graded $S$-modules, one considers finitely presented graded $S/I$-modules, where $I$ is the homogeneous ideal defining $X$.
Alternatively, one can work with the thick subcategory graded $S$-modules whose annihilators contain $I$; the essential image of the sheafification is then $\Coh X$.

\subsection{Coherent sheaves on toric varieties}

In this subsection we prove Theorem~\ref{thm:toric.computable} establishing the computability of the \nameft{Abel}ian category $\Coh X_\Sigma$ for a smooth toric variety $X_\Sigma$ with no torus factors (cf.\ \cite{CLS11} for notation).
Let $\grS=\CC[x_\rho \mid \rho \in \Sigma(1)]$ be the \nameft{Cox} ring of $X_\Sigma$.
This ring is graded by the divisor class group $\Cl X_\Sigma$.
Denote by $\Sfpgrmod$ the category of finitely presented graded $\grS$-modules.

\begin{theorem}[{\cite[Corollary~4.5]{BL_SerreQuotients}}] \label{thm:coh.equiv.toric.modules}
  Let $X_\Sigma$ be a toric variety with no torus factor.
  The exact and essentially surjective sheafification functor $\Sfpgrmod \to \Coh X_\Sigma, M \mapsto \widetilde{\grM}$ induces the equivalence
  \[
    \Sfpgrmod / \Sfpgrmod^0 \xrightarrow{\sim} \Coh X_\Sigma
  \]
   of categories where $\Sfpgrmod^0$ is defined as the kernel of the sheafification functor.
\end{theorem}

As $X_\Sigma$ is smooth the membership in the subcategory $\Sfpgrmod^0 \subset \Sfpgrmod$ is decidable.

\begin{prop}[{\cite[Cor.~3.6]{Cox95}, \cite[Prop.~5.3.10]{CLS11}}]\label{prop:decide_quasi_zero_toric}
  Let $X_\Sigma$ be a smooth toric variety with \nameft{Cox} ring $\grS$ and $\grM \in \Sfpgrmod$.
  Then $\widetilde{\grM} = 0$ if and only if some power of $B(\Sigma)$ is contained in the annihilator $\Ann_\grS \grM$, where $B(\Sigma) \unlhd \grS$ is the irrelevant ideal.\footnote{This is already false for simplicial nonsmooth toric varieties, cf.~\cite[Example~5.3.11]{CLS11} and \cite[Prop.~3.5]{Cox95}.}
  In particular, using well-known algorithms to compute annihilators and to test radical ideal membership (\nameft{Rabinowitsch} trick) one can decide the membership in $\Sfpgrmod^0$ for smooth toric varieties.
\end{prop}

Now we can prove Theorem~\ref{thm:toric.computable}, i.e., that the category $\Coh X_\Sigma$ is constructively \nameft{Abel}ian if $X_\Sigma$ is a smooth toric variety without torus factors.

\begin{proof}[Proof of Theorem~\ref{thm:toric.computable}]
  First, the category of finitely presented graded $\grS$-modules is constructively \nameft{Abel}ian by Corollary~\ref{coro:Sfpgrmod_computable}.
  Second, the membership in $\Sfpgrmod^0$ is decidable by Proposition~\ref{prop:decide_quasi_zero_toric}.
  Thus, by Theorem~\ref{thm:gabriel.computable} the category $\Sfpgrmod / \Sfpgrmod^0$ is constructively \nameft{Abel}ian.
  The statement follows from the equivalence of categories in Theorem~\ref{thm:coh.equiv.toric.modules}.
\end{proof}

In the case where all maximal cones of the smooth fan $\Sigma$ are full dimensional there is a more efficient way to decide membership in $\Sfpgrmod^0$.
Let $\sigma$ be a maximal cone, full dimensional by assumption.
Consider the inclusions $U_\sigma \subset X_\Sigma$, where $U_\sigma \cong \CC^n$.
  The inclusion map $\phi_\sigma: \CC^{\sigma(1)} \hookrightarrow \CC^{\Sigma(1)}$ sending $(a_\rho)_{\rho \in \sigma(1)}$ to $(b_\rho)_{\rho \in \Sigma(1)}$ with
  \[
    b_\rho= \begin{cases} a_\rho & \rho \in \sigma(1) \\ 1 & \mbox{otherwise} \end{cases}
  \]
induces the inclusion $U_\sigma \hookrightarrow X_\Sigma$ \cite[Prop.~5.2.10]{CLS11}.
Thus, to compute $\Gamma(U_\sigma,\widetilde{\grM}) = (M_{x^{\widehat{\sigma}}})_0$ as a module over $\Gamma(U_\sigma,\O_{X_\Sigma}) = (S_{x^{\widehat{\sigma}}})_0$, we only need to substitute $1$ for $x_\rho$ whenever $\rho \not\in \sigma(1)$ in the presentation matrix of $\grM$.
Then $\widetilde{\grM}$ is the zero sheaf if and only if $\Gamma(U_\sigma,\widetilde{\grM})=0$ for all maximal cones $\sigma \in \Sigma$.

For multigraded \nameft{Hilbert} polynomials and their relation to modules in $\Sfpgrmod^0$ we refer the reader to \cite[Lemma~2.12]{MS05}.

For nonsmooth toric varieties (with no torus factor) \nameft{S.~Gutsche} outlined in \cite{Gutsche_KernelOfSh} an algorithm to compute the $(S_{x^{\widehat{\sigma}}})_0$-module $(M_{x^{\widehat{\sigma}}})_0$ and hence to decide the membership in $\Sfpgrmod^0$ in the general case.

\begin{appendix}

\section{Inverse semigroups} \label{sec:isg}

In this section we recall some basic facts about inverse semigroups (cf.~\cite{Law98}).

A semigroup $H$ is a set equipped with a associative binary operation $\circ:H \times H \to H$.
It is called commutative\footnote{The adjective \nameft{Abel}ian is rarely used.} if $h \circ h' = h' \circ h$ for all $h,h' \in H$.
A monoid is a semigroup with a neutral element, i.e., a (necessarily unique) $n \in H$ satisfying $n \circ h = h = h \circ n$ for all $h \in H$.
An element $h \in H$ is called regular if it has at least one inverse, i.e., an element $y \in H$ satisfying $h \circ y \circ h=h$ and $y \circ h \circ y = y$.
It immediately follows that the two elements $h \circ y$ and $y \circ h$ are idempotents and each idempotent $e \in H$ is trivially of this form (set $h=y=e$).
A semigroup is called regular if all its elements are regular.
It can be shown that inverses in regular semigroups are unique if and only if all idempotents commute, providing two equivalent definitions of an \textbf{inverse semigroup}.
Denote the subset of idempotents by $E(H)$.
It is then the largest idempotent inverse (and hence commutative) subsemigroup, or equivalently, a meet-semilattice for the partial order $e \leq f :\!\!\iff e = e \wedge f$ defined by the meet $e \wedge f := e \circ f = f \circ e$.
It is bounded when $H$, and hence $E(H)$, is a monoid.
The unique inverse of $h$ is denoted by $h^{-1}$ (unless the semigroup is written additively then by $-h$).
The mapping $h \mapsto h^{-1}$ is an involution on $H$, i.e., $(h^{-1})^{-1} = h$ and $(h \circ g)^{-1} = g^{-1} \circ h^{-1}$.

Setting
\[
  h \leq g \iff h = g \circ h^{-1} \circ h \quad (\mbox{or equivalently } h = h \circ h^{-1} \circ g)
\]
defines a natural partial order on an inverse semigroup $H$ which extends the one on $E(H)$.
The meet-semilattice $E(H)$ is a singleton only if $H$ is a group.

A congruence $\sim$ on an (inverse) semigroup $H$ is an equivalence relation which respects the binary operation: $g \sim g'$ and $h \sim h'$ then $g \circ h \sim g' \circ h'$.
The set $H / \sim$ of equivalence classes (with binary operation induced by $\circ$) is again an (inverse) semigroup.
A congruence relation is called a group congruence if $H / \sim$ is a group.
The congruence $\sigma$ defined by
\[
  h \sim_\sigma g :\!\!\iff \exists x \in H: x \leq h,g
\]
is \textbf{the smallest\footnote{w.r.t.\  inclusion} group congruence} of the inverse semigroup $H$.
All idempotents are easily seen to be $\sigma$-equivalent but their $\sigma$-equivalence class $Z(H)\supseteq E(H)$ might contain non-idempotents, otherwise $\sigma$ is called idempotent pure.
In our applications $\sigma$ is far from being idempotent pure.

\begin{rmrk} \label{rmrk:zeroid_semirings}
  We call $Z(H)$ the \textbf{zeroid} of $H$ following \nameft{Bourne} and \nameft{Zassenhaus}, who first introduced it in \cite[Definition~4]{BZ58} as a distinguished two-sided ideal in a semiring with a commutative additive semigroup $H$.
  In that context it the smallest strongly closed ideal (or $h$-ideal) in the sense of \cite{Iiz59}; an additively inverse semiring\footnote{i.e., with a commutative inverse semigroup of addition} modulo its zeroid is a ring.
\end{rmrk}

In the proof of Proposition~\ref{prop:Zeroid_of_Gabriel} we need the following simple lemma:
\begin{lemm} \label{lemm:zeroid_characterization}
  Let $H$ is an inverse semigroup.
  Then
  \[
    Z(H) = \{ z \in H \mid \exists x \in H: x = x \circ z \} = \{ z \in H \mid \exists x \in H: x = z \circ x \} \mbox{.}
  \]
  Obviously, requiring $x \in E(H)$ does not alter the statement.
\end{lemm}
\begin{proof}
  For $z \in Z(H)$ we have $z \sim_\sigma e \in E(H)$, i.e., $\exists x \in H: x = x \circ x^{-1} \circ e$ and $x = x \circ x^{-1} \circ z$.
  The first equality implies $x \in E(H)$ and hence $x=x \circ z$ by the second.
 For the reverse inclusion let $x=x\circ z$ for $x \in E(H)$.
 Then $x= x \circ z = x \circ x^{-1} \circ z$, i.e., $x \leq x,z$ which means that $z \sim_\sigma x \in E(H)$.
\end{proof}

In the rest of this section we consider \emph{commutative} semigroups which we write additively.
A neutral element, in case it exists, will be denoted by $0$.

Note that a commutative regular semigroup is the same thing as a commutative inverse semigroup.
In this case
\[
  0(h) := h-h = -h + h \in E(H)
\]
is called \textbf{the idempotent associated to $h$}. More generally we define an action of $\Z$ as follows
\[
  z(h) = \begin{cases}
    \underbrace{h + \cdots + h}_z &, z \in \Z_{>0} \mbox{,} \\
    0(h) &, z = 0 \mbox{,} \\
    (-z)(-h) &, z \in \Z_{<0} \mbox{.}
  \end{cases}
\]

Hence, a commutative inverse semigroup $H$ defines a direct system of \nameft{Abel}ian groups
\[
  \left( H_e := \{ h \in H \mid 0(h) = e \} \right)_{e \in E(H)}
\]
over the meet-semilattice $E(H)$.
In particular, $H$ set-theoretically a disjoint union of \nameft{Abel}ian groups $H=\dot{\bigcup}_{e \in E(H)} H_e$.
The converse also holds as a special case of \cite[Theorem~4.11]{CP61}:
Any direct system of \nameft{Abel}ian groups over a meet-semilattice $L$ defines a commutative inverse semigroup $H$ with $E(H) \cong L$.

Denoting by $H_\infty := \varinjlim_{e \in E(H)} H_e$ the colimit \nameft{Abel}ian group we conclude:
\begin{coro} \label{coro:zeroid_is_kernel_of_colimit}
  The zeroid $Z(H)$ of the commutative inverse semigroup $H$ is the kernel of the canonical homomorphism $H \to H_\infty$.
  In particular, $H_\infty \cong H / Z(H)$.
\end{coro}

\begin{rmrk}
  The tensor product of commutative semigroups introduced in \cite{Gri69} only assumes bilinearity.
  This is enough for it to restrict to the full subcategory of commutative inverse semigroups: $-(h' \otimes h) = -h' \otimes h = h' \otimes (-h)$.
  
  However, it does \emph{not} restrict to the \emph{non}full subcategory of commutative monoids which has its own tensor product, characterizable as follows:
  The endofunctor $- \otimes H$ is the left adjoint of the (internal) $\Hom(H,-)$ functor, where $\Hom(H,H')$ of two commutative monoids is again a commutative monoid with pointwise multiplication as binary operation and neutral element given by the constant map to the neutral element of $H'$.
  Compared to the the tensor product for commutative semigroups it additionally satisfies $0_{H'} \otimes h = 0_{H'} \otimes 0_H= h' \otimes 0_H$, which implies that the latter is the neutral element of the tensor product monoid.
  The unit object of this closed symmetric monoidal structure is given by the natural numbers $(\N,+,0)$.
\end{rmrk}

There is, however, a \emph{non}closed symmetric monoidal structure on the full subcategory of commutative \emph{inverse} monoids defined as follows:
\begin{defn} \label{defn:tensor_cim}
  The tensor product of two commutative inverse monoids is the tensor product of their underlying commutative semigroups modulo the extra relations stating that $0_{H'} \otimes 0_H$ is the neutral element.\footnote{Now $0_{H'} \otimes h = 0_{H'} \otimes 0(h)$ and $h' \otimes 0_H=0(h') \otimes 0_H$ are idempotents but not necessarily neutral or equal.}
  The unit object of this tensor product is the \nameft{Abel}ian group $\Z$ and the isomorphisms $\Z \otimes H \cong H \cong H \otimes \Z$ are given by the above described action of $\Z$.
\end{defn}
This last tensor product has no right adjoint as it does not preserve the initial object (=zero object=trivial monoid): $\{0\} \otimes (\{0,\infty\},+) \cong (\{0,\infty\},+) \ncong \{0\}$.
Still, restricted to the full subcategory of \nameft{Abel}ian groups it yields the standard tensor product, where, again, $0_{H'} \otimes h = 0_{H'} \otimes 0_H= h' \otimes 0_H$ holds\footnote{\nameft{Proof}. $0_{H'} \otimes h + \underline{g' \otimes g} = (g'-g') \otimes h + \underline{g' \otimes g} = g' \otimes h - g' \otimes h + \underline{g' \otimes g} =g' \otimes (h-h+g) = \underline{g' \otimes g}$.}.

\section{Constructively \nameft{Abel}ian categories} \label{sec:ComputableAbelian}

In this section we sum up the notation from our previous paper \cite{BL}, which introduced the notion of a computable or constructively \nameft{Abel}ian category as a constructive setup for homological algebra.
If $\A$ is a constructively \nameft{Abel}ian category then any construction becomes algorithmic if it only depends on $\A$ being an \nameft{Abel}ian category.
See \cite{BaSF} for a non-trivial example, which details a construction of spectral sequences (of filtered complexes) only using the axioms of an \nameft{Abel}ian category.
However, this approach goes beyond developing algorithms for specific homological constructions; it is rather a framework which automatically turns \emph{any} construction only relying on the category being \nameft{Abel}ian\footnote{This can be done for several enrichments along the same lines \cite{BL_Sheaves,BL_ExtComputability}.} into an algorithm.

\begin{defn} \label{defn:computable}
  Let $\A$ be an \nameft{Abel}ian category.
  We say that $\A$ is \textbf{constructively \textbf{Abel}ian} (or \textbf{computable}) if all existential quantifiers and disjunctions appearing in the axioms can be turned into constructive ones.\label{defn:computable.Abelian}
\end{defn}

The following list only emphasizes the existential quantifiers and disjunctions of a \nameft{Abel}ian category, i.e., the axioms that need to be turned into constructive ones for computability.
Thus, we suppress the universal properties needed to correctly formulate some of the points below, as we assume that they are well-known to the reader.
A detailed treatise can be found, in the \texttt{arXiv} version of \cite[Appendix A]{BL}.

Let $M,N,M_1,M_2$ be objects in $\A$.\\
$\A$ is a \textbf{category}
\begin{enumerate}
  \item\label{Ab_IdentityMatrix} For any object $M$ there \emph{exists} an \textbf{identity morphism} $1_M$.
  \item\label{Ab_Compose} For any two composable morphisms $\phi,\psi$ there \emph{exists} a \textbf{composition} $\phi\psi$.
  \item\label{Ab_HomSetMember} $\Hom_\A(M,N)$ is a set with \emph{decidable} element membership.
  \item\label{Ab_HomSetEqual} Equality of two morphisms in $\Hom_\A(M,N)$ is \emph{decidable}.
\suspend{enumerate}
$\A$ is a category \textbf{enriched over commutative inverse monoids\footnote{Cf.~Definition~\ref{defn:tensor_cim}.}}:
\resume{enumerate}
  \item\label{Ab_AddMat} There \emph{exists} an \textbf{addition} $(\phi,\psi)\mapsto \phi+\psi$ in $\Hom_\A(M,N)$.
  \item\label{Ab_SubMat} There \emph{exists} a \textbf{subtraction} $(\phi,\psi)\mapsto \phi-\psi$ in $\Hom_\A(M,N)$.
  \item\label{Ab_ZeroMatrix} For all objects $M,N$ there \emph{exists} a \textbf{zero morphism} $0_{MN}$.
\suspend{enumerate}
$\A$ is a category \textbf{with zero}:
\resume{enumerate}
  \item There \emph{exists} a \textbf{zero object} $0$.
\suspend{enumerate}
$\A$ is a \textbf{pre-additive} category:
\resume{enumerate}
  \item\label{Ab_AbelianGroup} $\Hom_\A(M,N)$ with the above $(+,-,0_{MN})$ is an \nameft{Abel}ian group.
\suspend{enumerate}
$\A$ is an \textbf{additive} category:
\resume{enumerate}
  \item\label{Ab_DiagMatProd} There \emph{exists} a \textbf{product} $M_1\oplus M_2$ and \textbf{projections} $\pi_i: M_1\oplus M_2 \to M_i$ such that
  \item\label{Ab_UnionOfRows} for all pairs of morphisms $\psi_i:N\to M_i$, $i=1,2$ there \emph{exists} a unique \textbf{pairing} $\{\psi_1,\psi_2\}:N\to M_1\oplus M_2$.
\suspend{enumerate}
$\A$ is a \textbf{pre-\nameft{Abel}ian} category:
\resume{enumerate}
  \item\label{Ab_SyzygiesGenerators} For any morphism $\phi:M\to N$ there \emph{exists} a \textbf{kernel} $\ker \phi \stackrel{\kappa}{\hookrightarrow} M$, such that
  \item\label{Ab_DecideZeroEffectively} for any morphism $\tau:L\to M$ with $\tau\phi=0$ there \emph{exists} a unique \textbf{lift} $\liftofCalongB{\kappa}{\tau}: L\to \ker \phi$ of $\tau$ along $\kappa$, i.e., $(\liftofCalongB{\kappa}{\tau}) \kappa = \tau$.
  \item\label{Ab_UnionOfRows_IdentityMatrix} For any morphism $\phi:M\to N$ there \emph{exists} a \textbf{cokernel} $N \stackrel{\epsilon}{\twoheadrightarrow} \coker \phi$, such that
  \item\label{Ab_colift} for any morphism $\eta:N\to L$ with $\phi\eta=0$ there \emph{exists} a unique \textbf{colift} $\coliftofCalongB{\epsilon}{\eta}: \coker \phi \to L$ of $\eta$ along $\epsilon$, i.e., $\epsilon (\coliftofCalongB{\epsilon}{\eta}) = \eta$.
\suspend{enumerate}
$\A$ is an \textbf{\nameft{Abel}ian} category:
\resume{enumerate}
  \item\label{Ab_MonoLift} Any mono is the kernel of its cokernel, i.e., for any mono $\kappa:K \to M$ with cokernel epi $\phi$ and any morphism $\tau:L\to M$ with $\tau\phi=0$ there \emph{exists} a unique \textbf{lift} $\liftofCalongB{\kappa}{\tau}: L\to \ker \phi$ of $\tau$ along $\kappa$, i.e., $(\liftofCalongB{\kappa}{\tau}) \kappa = \tau$.
  \item\label{Ab_EpiColift} Any epi is the cokernel of its kernel, i.e., for any epi $\epsilon:N \to C$ with kernel mono $\phi$ and any morphism $\eta:N\to L$ with $\phi\eta=0$ there \emph{exists} a unique \textbf{colift} $\coliftofCalongB{\epsilon}{\eta}: \coker \phi \to L$ of $\eta$ along $\epsilon$, i.e., $\epsilon (\coliftofCalongB{\epsilon}{\eta}) = \eta$.
\end{enumerate}

\begin{exmp}
  The category of finitely generated torsion-free \nameft{Abel}ian groups is pre-\nameft{Abel}ian but not \nameft{Abel}ian.
  This follows as the mono $2:\Z \to \Z$ is not the kernel (which is the zero morphism $\Z \to 0$) of its cokernel.
  In particular, $2:\Z \to \Z$ is mono and epi but not iso.
\end{exmp}

\section{Example of a spectral sequence computation for sheaves} \label{sec:ex}
  
We demonstrate\footnote{using the packagnes of the \texttt{homalg} project \cite{homalg-project}; all computations can be reproduced on the \texttt{homalg}-online server for which only a web browser is required (\url{http:// homalg.math.rwth-aachen.de/}).} the difference between computations in the categories $\A=\Sfpgrmod$ and $\A/\C = \Sfpgrmod/\Sfpgrmod^0 \simeq \Coh \PP^n$ by computing the so-called \emph{grade filtration} of a graded $S$-module $\grM \in \A$ and of its sheafification $\shF = \widetilde{M} \in \A/\C$, respectively.
Codomains naturally appear in the latter computation.

Below we use the bidualizing spectral sequence to compute the grade filtration following \cite[Appendix B]{BaSF}.
\nameft{Quadrat}'s alternative approach to the grade filtration \cite{QGrade} is also general enough to allow the passage to quotient categories of coherent sheaves.

Consider the graded ring $\grS=\Q[x,y,z]$.
We define a graded $\grS$-module $\grM \in \A=\Sfpgrmod$ on $6$ generators satisfying $6$ relations given by the rows of the matrix \texttt{mat} below.

{\footnotesize
  \begin{Verbatim}[commandchars=!@\%,frame=single]
!color@blue%@gap>%!color@red%@ LoadPackage( "GradedModules" );%
true
!color@blue%@gap>%!color@red%@ Q := HomalgFieldOfRationalsInSingular( );;%
!color@blue%@gap>%!color@red%@ S := GradedRing( Q * "x,y,z" );;%
!color@blue%@gap>%!color@red%@ mat := HomalgMatrix( "[ \%
!color@blue%@>%!color@red%@ -x^2*z+x*y*z+x*z^2,y^2*z,-x*z+y*z,x-y,0,   0,   \%
!color@blue%@>%!color@red%@ -x^3+x^2*y+x^2*z,  x*y^2,-x^2+x*y,0,  x-y, -x*y,\%
!color@blue%@>%!color@red%@ 0,                 0,    0,       x*y,-y*z,0,   \%
!color@blue%@>%!color@red%@ 0,                 0,    0,       x^2,-x*z,0,   \%
!color@blue%@>%!color@red%@ 0,                 0,    0,       x*z,-z^2,0,   \%
!color@blue%@>%!color@red%@ 0,                 0,    0,       0,  0,   z    \%
!color@blue%@>%!color@red%@ ]", 6, 6, S );%
<A 6 x 6 matrix over a graded ring>
!color@blue%@gap>%!color@red%@ M := LeftPresentationWithDegrees( mat, [ 0, 0, 1, 2, 2, 1 ] );%
<A graded module presented by 6 relations for 6 generators>
  \end{Verbatim}
  }
  We define its sheafification $\shF=\widetilde{\grM}$ over $\PP^2=\Proj(S)$.
  {\footnotesize
  \begin{Verbatim}[commandchars=!@\%,frame=single]
!color@blue%@gap>%!color@red%@ LoadPackage( "Sheaves" );%
true
!color@blue%@gap>%!color@red%@ F := Sheafify( M );%
<A coherent sheaf on some 2-dimensional projective space>
  \end{Verbatim}
  }
   First we compute the grade filtration of the graded module $\grM$ induced by the (second quadrant) bidualizing spectral sequence
  \[
    {}^\mathrm{II} E^2_{pq}(\grM) := \Ext_\bullet^{-p}(\Ext_\bullet^q(\grM,S),S) \Longrightarrow \begin{cases} \grM, & \mbox{for } p+q=0 \\ 0, & \mbox{otherwise}, \end{cases}
  \]
  where $\Ext_{\bullet}$ denotes the graded $\Ext$ over $S$ (see~\cite[§9.1.3]{BaSF}). 
  Then we compute the grade filtration of the sheafification $\shF$ induced by
  \[
    {}^\mathrm{II} E^2_{pq}(\shF) := \mathcal{E}xt^{-p}(\mathcal{E}xt^q(\shF,S),S) \Longrightarrow \begin{cases} \shF, & \mbox{for } p+q=0 \\ 0, & \mbox{otherwise}, \end{cases}
  \]
  where $\mathcal{E}xt$ denotes the sheaf $\Ext$ over $\PP^2$.
  {\footnotesize
  \begin{Verbatim}[commandchars=!@\%,frame=single] 
!color@blue%@gap>%!color@red%@ II_E_M := BidualizingSpectralSequence( M );%
<A stable homological spectral sequence with pages at levels [ 0 .. 4 ]
  each consisting of graded modules at bidegrees [ -3 .. 0 ]x[ 0 .. 3 ]>
!color@blue%@gap>%!color@red%@ II_E_F := BidualizingSpectralSequence( F );%
<A stable homological spectral sequence with pages at levels [ 0 .. 3 ]
  each consisting of sheaves at bidegrees [ -3 .. 0 ]x[ 0 .. 3 ]>
  \end{Verbatim}
  }
  The \texttt{Display} command prints the board of the corresponding collapsing first spectral sequence ${}^\mathrm{I} E_{pq}$ consisting of the three pages ${}^\mathrm{I} E_{pq}^0,{}^\mathrm{I} E_{pq}^1,{}^\mathrm{I} E_{pq}^2$.
  Then it prints the board of the second spectral sequence ${}^\mathrm{II} E_{pq}$, i.e., the bidualizing spectral sequence.
  For ${}^\mathrm{II} E_{pq}(\grM)$ we have five pages, which is the maximal possible number of pages.
  For ${}^\mathrm{II} E_{pq}(\shF)$ the spectral sequence has four pages, which is again the maximal possible number.
  We display their boards\footnote{Legend: \texttt{*} $\triangleq$ nonzero object, \texttt{s} $\triangleq$ nonzero object which has stabilized, and \texttt{.} $\triangleq$ zero object.} side by side.
  {\footnotesize
  \begin{Verbatim}[commandchars=!@\%,frame=single] 
!color@blue%@gap>%!color@red%@ Display( II_E_M );%         !color@blue%@gap>%!color@red%@ Display( II_E_F );%
The associated transposed spectral sequence:
a homological spectral sequence at bidegrees [ [ 0 .. 3 ], [ -3 .. 0 ] ]
---------                       ---------
Level 0:                        Level 0:
                                
 * * * *                         * * * *
 * * * *                         * * * *
 * * * *                         * * * *
 . * * *                         . * * *
---------                       ---------
Level 1:                        Level 1:
                                
 * * * *                         * * * *
 . . . .                         . . . .
 . . . .                         . . . .
 . . . .                         . . . .
---------                       ---------
Level 2:                        Level 2:
                                
 s . . .                        s . . .
 . . . .                        . . . .
 . . . .                        . . . .
 . . . .                        . . . .

Now the spectral sequence of the bicomplex:
a homological spectral sequence at bidegrees [ [ -3 .. 0 ], [ 0 .. 3 ] ]
---------                       ---------
Level 0:                        Level 0:

 * * * *                        * * * *
 * * * *                        * * * *
 * * * *                        * * * *
 . * * *                        . * * *
---------                       ---------
Level 1:                        Level 1:

 * * * *                        * * * *
 * * * *                        * * * *
 * * * *                        * * * *
 . . * *                        . . * *
---------                       ---------
Level 2:                        Level 2:

 s . . .                        . . . .
 * . . .                        . . . .
 * * * .                        . * s .
 . . * *                        . . . *
---------                       ---------
Level 3:                        Level 3:

 s . . .                        . . . .
 * . . .                        . . . .
 . . s .                        . . s .
 . . . *                        . . . s
---------           
Level 4:            

 s . . .           
 . . . .           
 . . s .           
 . . . s           

\end{Verbatim}
}

Our algorithms also provide the isomorphisms from the filtered sheaf to $\shF$.
Let
\[
  \grN :=
  \coker \left(\begin{array}{ccccc|c}
\cdot&    \cdot&     \cdot&       x& -z& \cdot \\  
x y z&-x z^2&-x z+y z&-y&z&  \cdot\\
x^2 y&-x^2 z&-x^2+x y&\cdot& x-y&-x y\\
\hline
\cdot&    \cdot&     \cdot&       \cdot& \cdot&  z
\end{array}\right)
\]
be the graded $\grS$-module with the above matrix of relations for $6$ abstract generators of degrees $[ 0, 0, 1, 2, 2, 1 ]$.
The isomorphism $\widetilde{\grN}\to\shF$ is given by the \nameft{Gabriel} morphism $(1_\grN, m, \jmath)$, where $\jmath$ is the cokernel epi of $\beta:S \xrightarrow{\left(\begin{smallmatrix} \cdot&\cdot&\cdot&x&-z&\cdot \end{smallmatrix}\right)} \grM$ and $m:N\to\coker\beta$ is represented by the matrix
\[
  m =
  \begin{pmatrix}
\cdot& -1&\cdot& \cdot& \cdot& \cdot \\
1& \cdot& \cdot& \cdot& \cdot& \cdot  \\
-x&-y&-1&\cdot& \cdot& \cdot  \\
\cdot& \cdot& \cdot& -1&\cdot& \cdot  \\
\cdot& \cdot& \cdot& \cdot& -1&\cdot  \\
\cdot& \cdot& \cdot& \cdot& \cdot& -1
  \end{pmatrix} \mbox{.}
\]

\nameft{Grothendieck} remarks \cite[p.~19]{Tohoku_en}:
\begin{quote}
It is particularly convenient to use, when we have a spectral sequence (cf.~2.4) in $\A$, the fact that some terms of the spectral sequence belong to $\C$: reducing $\bmod$ $\C$ (i.e.\  applying the functor $\QQ$), we find a spectral sequence in $\A/\C$ in which the corresponding terms vanish, whence we have exact sequences $\bmod$ $\C$, with the help of the usual criteria for obtaining exact sequences from spectral sequences in which certain terms have vanished.
\end{quote}

\nameft{Gabriel} morphisms allow us to compute directly with sheaves (in $\A/\C$) rather than with graded modules in $\A=\Sfpgrmod$.
The sheaf spectral sequence ${}^\mathrm{II} E_{pq}(\shF)$ stabilizes on page earlier than the module one ${}^\mathrm{II} E_{pq}(\grM)$; the nonzero graded modules on the ${}^\mathrm{II} E^3_{pq}(\grM)$ which did not yet stabilize all lie in $\C=\Sfpgrmod^0$.
We can say that computing the above codomain $\jmath$ corresponds to the last page of ${}^\mathrm{II} E_{pq}(\grM)$.

It would be interesting to know if there exist algorithms to compute ``resolutions modulo $\C$'' of graded modules in $\A$, i.e., where only the sheafification of such a resolution is required to be acyclic.
Such potentially shorter resolutions might be cheaper to compute than $\A$-resolutions.

\section{The computability of \nameft{Gabriel} localizations} \label{sec:reflective}

The \nameft{Gabriel} morphism approach has advantages over two other approaches we present in this appendix.
The language of \nameft{Gabriel} morphisms establishes the computability of $\A/\C$ (cf.\ Definition~\ref{defn:computable}) by realizing the colimit as computing modulo the zeroid.
The other two approaches require the thick subcategory $\C$ to be localizing and that the corresponding localization monad is computable.
However, even if a localization monad exists and is computable, as the case in our application to coherent sheaves, computing the monad seems to be expensive on average.

\subsection{Preliminaries on \texorpdfstring{\nameft{Gabriel}}{Gabriel} localizations} 

In this subsection we recall some results about \nameft{Gabriel} localizations.
See \cite{Gab_thesis,BL_Monads} for details.

Let $\A$ be an \nameft{Abel}ian category and $\C \subset \A$ a thick subcategory.
An object $M \in \A$ is called \textbf{$\C$-saturated} if $\Ext^i(C,M)=0$ for $i=0,1$ and all $C \in \C$.
Denote by $\Sat_\C(\A)$ the full subcategory of $\C$-saturated objects.

The thick subcategory $\C$ is called \textbf{localizing} if the canonical functor $\QQ: \A \to \A/\C$ admits a right adjoint $\SS:\A/\C \to \A$, called the \textbf{section functor} of $\QQ$.
It is fully faithful, left exact, and preserves products.

Denote by $\delta: \QQ \circ \SS \to \Id_{\A/\C}$ and $\eta:\Id_\A \to \SS \circ \QQ$ the counit and unit of the adjunction $\QQ \dashv \SS$, respectively.
The counit $\delta$ is a natural isomorphism in our setup.
We call such a canonical functor $\QQ$ a \textbf{\nameft{Gabriel} localization} and the associated monad
\[
  (\WW,\eta,\mu):=(\SS \circ \QQ,\eta,\mu=\SS\delta \QQ)
\]
the \textbf{\nameft{Gabriel} monad}.
An object $M$ in $\A$ is $\C$-saturated if and only if $\eta_M:M \to \WW(M)$ is an isomorphism.
The essential image of $\SS$ is $\Sat_\C(\A)$.
In particular, $\Sat_\C(\A) \simeq \SS(\A/\C) \simeq \A/\C$ is an \nameft{Abel}ian category.

Define the \textbf{saturation functor} $\widehat{\QQ} := \cores_{\Sat_\C(\A)} \WW: \A \to \Sat_\C(\A)$.
The adjunction $\widehat{\QQ} \dashv (\iota: \Sat_\C(\A) \hookrightarrow \A)$ corresponds under the above equivalence to the adjunction $\QQ \dashv (\SS: \A/\C \to \A)$.
They both share the same adjunction monad $\WW = \SS \circ \QQ = \iota \circ \widehat{\QQ}: \A \to \A$.
In particular, $\widehat{\QQ}$ is exact and $\iota$ preserves kernels and products.

\subsection{The computability of \texorpdfstring{$\Sat_\C(\A)$}{Sat\_C(A)}} \label{subsec:Sat}

In case $\C$ is a localizing subcategory of the \nameft{Abel}ian category $\A$ then by $\A/\C\simeq\Sat_\C(\A)$ the computability of $\A/\C$ would follow from that of $\Sat_\C(\A)$.
This subsection shows that the computability of $\Sat_\C(\A)$ is implied by the computability of the \nameft{Gabriel} monad $\WW$.

\begin{defn}
  We call a monad $(\WW,\eta,\mu)$ \textbf{computable} if the underlying functor together with the unit $\eta$ and the multiplication $\mu$ are computable\footnote{
    If $(\WW,\eta,\mu)$ is computable then we can decide membership in $\C$.
    If $\mu: \WW^2 \to \WW$ is an isomorphism then $\eta \WW = \WW \eta$ is the (unique) inverse of $\mu$  by the coherence conditions.
    In this case the computability of $\mu$ follows from that of $\WW$ and $\eta$.
  }.
\end{defn}

$\Sat_\C(\A) \simeq \WW(\A)$ are not in general \nameft{Abel}ian \emph{sub}categories of $\A$ (in the sense of \cite[p.~7]{weihom}), as the embedding functor $\iota: \Sat_\C(\A) \to \A$ is only left exact.
Thus, all constructions in the (full) replete subcategory $\Sat_\C(\A) \subset \A$ coincide with those in $\A$ except for cokernels.
The following proposition resolves this problem with the cokernel.

\begin{prop} \label{prop:comp_using_Sat}
  Let $\A$ be a constructive \nameft{Abel}ian category and $\C$ a localizing subcategory.
  If the \nameft{Gabriel} monad $(\WW,\eta,\mu)$ is computable then $\Sat_\C(\A) \simeq \A/\C$ is constructively \nameft{Abel}ian.
\end{prop}
\begin{proof}
  The full embedding functor $\iota:\Sat_\C(\A) \hookrightarrow \A$ preserves kernels and products.
  \journal{
  It remains to treat cokernels (and their colifts).
  Let $M \xrightarrow{\phi} N$ be a morphism in $\Sat_\C(\A)$.
  Then $\coker_{\Sat_\C(\A)}\phi \cong \widehat{\QQ}(\coker_\A \iota(\phi))$ with cokernel epi $\epsilon_{\Sat_\C(\A), \phi}: N \to \coker_{\Sat_\C(\A)} \phi$ given by the composition of the $\A$-morphisms
  \[
    \iota(N) \xrightarrow{\epsilon_{\A, \iota(\phi})} \coker_\A \iota(\phi) \xrightarrow{\eta_{\coker_\A  \iota(\phi)}} \WW(\coker_\A \iota(\phi)) \cong \iota(\coker_{\Sat_\C(\A)}\phi) \mbox{,}
  \]
  where $\epsilon_{\A,\iota(\phi)}$ is the cokernel epi of $\iota(\phi)$ in $\A$ and $\eta$ is the unit of the monad $\WW$.
  In particular, the colift in $\Sat_\C(\A)$ arises from applying the saturation functor $\widehat{\QQ}$ to the colift in $\A$.
  }{
  It remains to treat cokernels (and their colifts) which is the content of Lemma~\ref{lemm:cokernel}.
  }
\end{proof}

\journal{}{
\begin{lemm} \label{lemm:cokernel}
Let $M \xrightarrow{\phi} N$ be a morphism in the category $\Sat_\C(\A)$.
Then\footnote{We write $\coker_\A \phi$ for $\coker_\A \iota(\phi)$.} $\coker_{\Sat_\C(\A)}\phi \cong \widehat{\QQ}(\coker_\A \phi)$ with the cokernel epi $\epsilon_{\Sat_\C(\A), \phi}: N \to 
\coker_{\Sat_\C(\A)} \phi$ given by the composition of the $\A$-morphisms
\[
  \iota(N) \xrightarrow{\epsilon_{\A, \phi}} \coker_\A \phi \xrightarrow{\eta_{\coker_\A  \phi}} \WW(\coker_\A \phi) \cong \iota(\coker_{\Sat_\C(\A)}\phi) \mbox{,}
\]
where $\epsilon_{\A,\phi}$ is the cokernel epi of $\phi$ in $\A$ and $\eta$ is the unit of the monad $\WW$.
In particular, the colift in $\Sat_\C(\A)$ arises from applying the saturation functor $\widehat{\QQ}$ to the colift in $\A$.
\end{lemm}
\begin{proof}
Consider the exact sequence $\iota(M) \xrightarrow{\iota(\phi)} \iota(N) \xrightarrow{\epsilon_{\A,\phi}} \coker_\A \phi \to 0$ in $\A$.
The exactness of $\widehat{\QQ}$ yields an exact sequence $\WW(M) \xrightarrow{\WW(\phi)} \WW(N) \xrightarrow{\widehat{\QQ}(\epsilon_{\A,\phi})} \widehat{\QQ}(\coker_\A \phi) \to 0$ in $\Sat_\C(\A)$ (which might not be exact in $\A$).
The counit of the adjunction $\widehat{\delta}: \widehat{\QQ} \circ \iota \stackrel{\sim}{\to} \Id_{\Sat_\C(\A)}$ yields the following commutative diagram
\begin{center}
  \begin{tikzpicture}
  
    \node (QiM) {$\WW(M)$};
    \node (QiN) at (4,0) {$\WW(N)$};
    \node (Qcoker) at (8,0) {$\widehat{\QQ}(\coker_\A \phi)$};
    \node (Q0) at (11,0) {$0$};
    \node (M) at (0,-1.5) {$M$};
    \node (N) at (4,-1.5) {$N$};
    \node (cokerSat) at (8,-1.5) {$\coker_{\Sat_\C(\A)} \phi$};
    \node (0) at (11,-1.5) {$0$};
    
    \draw[-stealth'] (QiM) -- node[above] {$\WW(\phi)$} (QiN);
    \draw[-stealth'] (QiN) -- node[above] {$\widehat{\QQ}(\epsilon_{\A,\phi})$} (Qcoker);
    \draw[-stealth'] (Qcoker) -- (Q0);
    \draw[-stealth'] (M) -- node[above] {$\phi$} (N);
    \draw[-stealth'] (N) -- node[above] {$\epsilon_{\Sat_\C(\A),\phi}$} (cokerSat);
    \draw[-stealth'] (cokerSat) -- (0);
    \draw[-stealth'] (QiM) -- node[left] {$\widehat{\delta}_M$} node[right] {$\cong$} (M);
    \draw[-stealth'] (QiN) -- node[left] {$\widehat{\delta}_N$} node[right] {$\cong$} (N);
    \draw[dotted,-stealth'] (Qcoker) -- node[right] {$\cong$} (cokerSat);
    
  \end{tikzpicture}
\end{center}
The natural isomorphism $\coker_{\Sat_\C(\A)} \phi \cong \widehat{\QQ}(\coker_\A \phi)$ now follows from the essential uniqueness of the cokernel.

The last assertion follows from the triangle identity $(\eta\iota)(\iota \widehat{\delta})=\Id_\iota$ and the commutation rules $\eta_{\iota(N)} \WW(\epsilon_{\A,\phi}) = \epsilon_{\A,\phi} \eta_{\coker_\A \phi}$, expressing that $\eta$ is a natural transformation.

\begin{center}
  \begin{tikzpicture}[scale=0.92]
    
    \coordinate (right) at (0.31,0);
    \coordinate (left) at ($(0,0)-(right)$);
        
    \node (QiM) {$\WW(\iota(M))$};
    \node (QiN) at (5.5,0) {$\WW(\iota(N))$};
    \node (Qcoker) at ($2*(QiN)$) {$\WW(\coker_\A \phi)$};
    \node (Q0) at ($(Qcoker)+(3,0)$) {$0$};
    \node (M) at (0,-2) {$\iota(M)$};
    \node (N) at ($(M)+(QiN)$) {$\iota(N)$};
    \node (cokerSat) at ($(M)+(Qcoker)$) {$\iota(\coker_{\Sat_\C(\A)} \phi)$};
    \node (0) at ($(M)+(Q0)$) {$0$};
    
    \draw[-stealth'] (QiM) -- node[above] {$\WW(\iota(\phi))$} (QiN);
    \draw[-stealth'] (QiN) -- node[above] {$\WW(\epsilon_{\A,\phi})$} (Qcoker);
    \draw[-stealth'] (Qcoker) -- (Q0);
    \draw[-stealth'] (M) -- node[above] {$\iota(\phi)$} (N);
    \draw[-stealth'] (N) -- node[above] {$\iota(\epsilon_{\Sat_\C(\A),\phi})$} (cokerSat);
    \draw[-stealth'] (cokerSat) -- (0);
    
    \draw[-stealth'] ($(QiM.south)+(left)$) -- node[left] {$(\iota \widehat{\delta})_M$} node[right] {$\cong$} ($(M.north)+(left)$);
    \draw[-stealth'] ($(QiN.south)+(left)$) -- node[left] {$(\iota \widehat{\delta})_N$} node[right] {$\cong$} ($(N.north)+(left)$);
    \draw[dotted,-stealth'] ($(Qcoker.south)+(left)$) -- node[right] {$\cong$} ($(cokerSat.north)+(left)$);
    
    \draw[stealth'-] ($(QiM.south)+(right)$) -- node[right] {$(\eta \iota)_M$} ($(M.north)+(right)$);
    \draw[stealth'-] ($(QiN.south)+(right)$) -- node[right] {$(\eta \iota)_N$} ($(N.north)+(right)$);
    \draw[dotted,stealth'-] ($(Qcoker.south)+(right)$) -- ($(cokerSat.north)+(right)$);
    
  \end{tikzpicture}
  \qedhere
\end{center}
\end{proof}
}

\subsection{\nameft{Gabriel}-\nameft{Zisman} localization} \label{subsec:GZ}

The concept of \nameft{Gabriel} localizations is a special case of \nameft{Gabriel}-\nameft{Zisman} localizations (cf.\ \cite[Prop.~I.1.3]{GabZis}).
We describe the special case, which relies on the adjunction $\QQ\dashv(\SS:\A/\C\to\A)$.
For the general case see \cite[Chap.~1.2]{GabZis}.

They use the unit of the adjunction $\eta:\Id_\A \to \WW$ and the $\Hom$-adjunction
\[
 \Phi_{M,N}=\Phi_{\QQ(M),\QQ(N)}: \Hom_{\A/\C}(\QQ(M),\QQ(N)) \cong \Hom_\A(M, \WW(N)) \mbox{,}
\]
to describe a morphism $\psi:\QQ(M) \to \QQ(N)$ in $\A/\C$ as $M\xrightarrow{\Phi_{M,N}(\psi)}\WW(N)\xleftarrow{\eta_{N}}N$, a right fraction of two morphisms in $\A$.

One can use the $2$-arrow calculus of right fractions for the computability of $\A/\C$.
We suppress the details.
\begin{coro} \label{coro:garbiel.zisman}
  Let $\A$ be a constructively \nameft{Abel}ian category.
  If the \nameft{Gabriel} monad $(\WW,\eta,\mu)$ is computable then $\A/\C$ (with above the calculus of right fractions) is constructively \nameft{Abel}ian.
\end{coro}

\subsection{Comparison of the computational approaches}

Both the computability of $\Sat_\C(\A)$ from Proposition~\ref{prop:comp_using_Sat} and the \nameft{Gabriel}-\nameft{Zisman} localization from Corollary~\ref{coro:garbiel.zisman} rely on saturation.
For the computability of $\Sat_\C(\A)$ we need to saturate the input and all cokernels (taken in $\A$), and in \nameft{Gabriel}-\nameft{Zisman}'s calculus of right fractions we need to compute the unit of the adjunction $\eta_N: N \to \WW(N)$ for every object $N$ occurring as target of a morphism.
In our applications to coherent sheaves both the saturation and the computation of the unit of the adjunction seem to be expensive.
In particular, the approach of \nameft{Gabriel} morphisms seems best suited for an efficient implementation.

\journal{}{
For the sake of completeness, we describe how to convert a morphism from one computational model to another.

We can associate to a \nameft{Gabriel} morphism $\phi=[\imath_\phi,\overunderline{\phi},\jmath_\phi]:M\to N$ in $\A$ with respect to $\C$ a morphism in $\Sat_\C(\A)$ by applying the saturation functor $\widehat{\QQ}:\A \to \Sat_\C(\A)$ to $\phi$ which turns $\imath_\phi$ and $\jmath_\phi$ into isomorphisms.
Thus we can invert them and get
\[
  \widehat{\QQ}(\phi):=\widehat{\QQ}(\imath_\phi)^{-1}\widehat{\QQ}(\overunderline{\phi})\widehat{\QQ}(\jmath_\phi)^{-1}: \widehat{\QQ}(M)\to \widehat{\QQ}(N)\mbox{.}
\]
Conversely, any morphism $\widehat\psi:\widehat{\QQ}(M)\to\widehat{\QQ}(N)$ yields a \nameft{Gabriel} morphism from $M$ to $N$ defined as the lift $\liftofCalongB{\eta_N}{(\eta_M \widehat{\psi})}$ of $\eta_M \widehat{\psi}$ along $\eta_N$, both regarded as honest \nameft{Gabriel} morphisms (cf.~Lemma~\ref{lemm:lift_of_gabriel_morphisms}).


To any morphism $M \xrightarrow{\phi} \WW(N) \xleftarrow{\eta_{N}} N$ in $\A/\C$ described by a right fraction we can associate the morphism $\widehat{\phi}:= \WW(\phi)(\eta_{\WW(N)})^{-1}: \WW(M) \to \WW(N)$ in $\WW(\A)$.
Conversely, to any morphism $\widehat{\phi}: \WW(M) \to \WW(N)$ in $\WW(\A)$ we can associate the morphism $M \xrightarrow{\phi} \WW(N) \xleftarrow{\eta_{N}} N$ defined by $\phi := \eta_M \widehat{\phi}$.

To any \nameft{Gabriel} morphism $[\imath,\varphi,\jmath]$ we can associate the right fraction
\[
  M\xrightarrow{\eta_M \WW(\imath)^{-1} \WW(\varphi) \WW(\jmath)^{-1}} \WW(N)\xleftarrow{\eta_N}N\mbox{.}
\]
Conversely, any right fraction $M\xrightarrow{\varphi}\WW(N)\xleftarrow{\eta_N}N$ yields a \nameft{Gabriel} morphism from $M$ to $N$ defined as the lift $\liftofCalongB{\eta_N}{\varphi}$ of $\varphi$ along $\eta_N$, both regarded as honest \nameft{Gabriel} morphisms (cf.~Lemma~\ref{lemm:lift_of_gabriel_morphisms}).
}

\end{appendix}

\section*{Acknowledgments}
We thank \nameft{Sebastian Posur} for helpful discussions about the subtleties of the enrichment of $\Gen(\A)$.

\def\cprime{$'$} \def\cprime{$'$} \def\cprime{$'$} \def\cprime{$'$}
  \def\cprime{$'$}

\end{document}